\documentclass[12pt]{amsart}
\usepackage{fullpage}
\usepackage{amsmath}
\usepackage{graphics}
\usepackage[all]{xy}
\xyoption{v2} \xyoption{knot}

\let\url=\undefined
\usepackage
[bookmarks,colorlinks,linkcolor=blue,citecolor=blue,urlcolor=blue] {hyperref}

\title{Convex Hull Realizations of the Multiplihedra }
\author {Stefan Forcey}

\thanks{Thanks to {\Xy-pic} for the diagrams. }
\email{sforcey@tnstate.edu}
\address{Department of Physics and Mathematics\\
Tennessee State University\\
Nashville, TN 37209 \\
USA\\
}


\swapnumbers \theoremstyle{plain}
\newtheorem{theorem}{Theorem}[section]
\newtheorem{lemma}[theorem]{Lemma}

\theoremstyle{definition}
\newtheorem{definition}[theorem]{Definition}
\newtheorem{example}[theorem]{Example}

\theoremstyle{remark}
\newtheorem{remark}[theorem]{Remark}

\def\cal#1{\mathcal{#1}}

\begin{document}

\begin{abstract}
We present a simple algorithm for determining the extremal points in Euclidean space whose convex
hull  is the $n^{th}$ polytope in the sequence known as the multiplihedra. This answers the open
question of whether the multiplihedra could be realized as convex polytopes. We use this
realization to unite the approach to $A_n$-maps of Iwase and Mimura to that of Boardman and Vogt.
We include a review of the appearance of the $n^{th}$ multiplihedron for various $n$ in the studies
of higher homotopy commutativity, (weak) $n$-categories, $A_{\infty}$-categories, deformation
theory, and moduli spaces. We also include suggestions for the use of our realizations in some of
these areas as well as in related studies, including enriched category theory and the graph
associahedra.
$$
\xy 0;/r.45pc/:
   (-4,3)*=0{}="a1";
  (4,3)*=0{}="a2";
  (6,-2)*=0{}="a3";
  (0,-7)*=0{}="a4";
  (-6,-2)*=0{}="a5";
  (-6,6.5)*=0{}="b1";
  (6,6.5)*=0{}="b2";
  (10,-4)*=0{}="b3";
  (0,-12)*=0{}="b4";
  (-10,-4)*=0{}="b5";
  (-8,10.5)*=0{}="c1";
  (8,10.5)*=0{}="c2";
  (13,7.5)*=0{}="c3";
  (16,0)*=0{}="c4";
  (14,-6)*=0{}="c5";
  (4,-14)*=0{}="c6";
  (0,-16)*=0{}="c7";
  (-4,-14)*=0{}="c8";
  (-14,-6)*=0{}="c9";
  (-16,0)*=0{}="c10";
  (-13,7.5)*=0{}="c11";
 (-9,-1.4)*{\hole}="x"; (-9,-1.4)*{\hole}="x";
(-6.8,4.4)*{\hole}="y";
  (9,-1.4)*{\hole}="w";
  (0,-12)*{\hole}="v";
 (6.8,4.4)*{\hole}="z";
  "a1";"a2" \ar@{-}; "a1";"a2"\ar@{-};
  "a2";"a3"  \ar@{-};
    "a3";"a4" \ar@{-};
  "a4";"a5" \ar@{-};
  "a5";"a1" \ar@{-};
 "b1";"b2" \ar@{-};
 "b2";"b3" \ar@{-};
 "b3";"b4" \ar@{-};
 "b4";"b5" \ar@{-};
 "b5";"b1" \ar@{-};
 "c1";"c2"\ar@{-};
  "c2";"c3"  \ar@{-};
    "c3";"c4" \ar@{-};
  "c4";"c5" \ar@{-};
"c5";"c6"\ar@{-};
  "c6";"c7"  \ar@{-};
    "c7";"c8" \ar@{-};
  "c8";"c9" \ar@{-};
"c9";"c10"\ar@{-};
  "c10";"c11"  \ar@{-};
  "c11";"c1"  \ar@{-};
"b1";"c1"\ar@{-}; "b2";"c2"\ar@{-}; "b3";"c5"\ar@{-};
  "b4";"c6"  \ar@{-};
  "b4";"c8"  \ar@{-};
  "b5";"c9"  \ar@{-};
"a1";"y"  \ar@{-}; "y";"c11"  \ar@{-}; "a2";"z"  \ar@{-}; "z";"c3"  \ar@{-}; "a3";"w"  \ar@{-};
"w";"c4"  \ar@{-}; "a4";"v"  \ar@{-}; "v";"c7"  \ar@{-}; "a5";"x"  \ar@{-}; "x";"c10"  \ar@{-};
     \endxy
$$
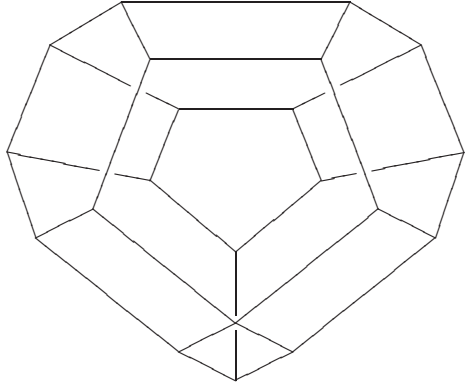
Figure 1: The main character: the 3-d multiplihedron $\cal{J}(4).$
\end{abstract}

\maketitle \tableofcontents
\section{Introduction}
The associahedra are the famous sequence of polytopes denoted $\cal{K}(n)$ from \cite{Sta} which
characterize the structure of weakly associative products. $\cal{K}(1) = \cal{K}(2) = $ a single
point, $\cal{K}(3)$ is the line segment, $\cal{K}(4)$ is the pentagon, and $\cal{K}(5)$ is the
following 3d shape:
$${\cal K}(5) = 
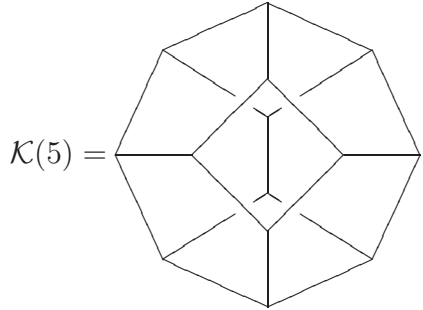
\xy 0;/r.3pc/:
   (0,16)*=0{}="a";
  (11,11)*=0{}="b";
  (16,0)*=0{}="c";
  (11,-11)*=0{}="d";
  (0,-16)*=0{}="e";
  (-11,-11)*=0{}="f";
  (-16,0)*=0{}="g";
  (-11,11)*=0{}="h";
  (0,8)*=0{}="2a";
  (8,0)*=0{}="2b";
  (0,-8)*=0{}="2c";
  (-8,0)*=0{}="2d";
(0,4)*=0{}="3a";
  (0,-4)*=0{}="3b";
 (-2.44,5.56)*{\hole}="p";
 (2.44,5.56)*{\hole}="q";
 (-2.44,-5.56)*{\hole}="r";
 (2.44,-5.56)*{\hole}="s";
  "a";"b" \ar@{-}; "a";"b"\ar@{-};
  "b";"c"  \ar@{-};
    "c";"d" \ar@{-};
  "d";"e" \ar@{-};
 "e";"f" \ar@{-};
 "f";"g" \ar@{-};
 "g";"h" \ar@{-};
 "h";"a" \ar@{-};
 "2a";"2b"\ar@{-};
  "2b";"2c"  \ar@{-};
    "2c";"2d" \ar@{-};
  "2d";"2a" \ar@{-};
"2a";"a"\ar@{-};
  "2b";"c"  \ar@{-};
    "2c";"e" \ar@{-};
  "2d";"g" \ar@{-};
"3a";"p"\ar@{-}; "p";"h"\ar@{-};
  "3b";"r"  \ar@{-};
  "r";"f"  \ar@{-};
"3a";"q"\ar@{-}; "q";"b"\ar@{-}; "3b";"s"\ar@{-};
  "s";"d"  \ar@{-};
  "3b";"3a"  \ar@{-};
%
%
     \endxy
$$
The original examples of weakly associative product
 structure are the $A_n$ spaces, topological $H$-spaces with weakly associative
 multiplication of points. Here ``weak'' should be understood as ``up to homotopy.'' That is,
 there is a path in the space
 from $(ab)c$ to $a(bc).$ An $A_{\infty}$-space $X$ is characterized by its admission of an action
 $$\cal{K}(n) \times X^n \to X$$
 for all $n.$

 Categorical examples begin with the monoidal categories as defined in \cite{MacLane},
 where there is a
 weakly associative tensor product of objects. Here ``weak'' officially means ``naturally isomorphic.''
 There is a natural
 isomorphism $\alpha: (U\otimes V) \otimes W \to U\otimes (V \otimes W).$

 The complexes now known as the multiplihedra, usually denoted $\cal{J}(n)$,
  were first pictured by Stasheff, for $n \le 4$ in
\cite{sta2}. The $n^{th}$ multiplihedron as a complex can be seen as a subdivision of the complex
$\cal{K}(n) \times I.$ Indeed the drawing of $\cal{J}(4)$ in \cite{sta2} appears as a pentagonal
cylinder. The drawing in Figure 1 of this paper can be seen as a view of that cylinder from below.
In \cite{sane2} the authors give an alternative definition of $\cal{J}(n)$ based on the subdivision
of the cylinder with $\cal{K}(n)$ base.

 The multiplihedra were introduced in order to approach a full description of the category of
$A_{\infty}$ spaces by providing the underlying structure for morphisms which preserved the
structure of the domain space ``up to homotopy'' in the range. Recall that an $A_{\infty}$ space
itself is a monoid only ``up to homotopy.''  Thus the multiplihedra are used to recognize the
$A_{\infty}$ (as well as $A_n$) maps. Stasheff described how to construct the 1-skeleton of these
complexes, but stopped short of a full combinatorial description.

In \cite{BV1} Boardman and Vogt take up the challenge of a complete description of the category of
$A_{\infty}$ spaces and maps (and their $A_n$ versions.) Their approach is to use sequences of
spaces of binary trees with interior edges given a length in $[0,1]$. They show that the space of
such trees  with $n$ leaves (under certain equivalence relations regarding length zero edges) is
precisely the $n^{th}$ associahedron. They then develop several homotopy equivalent versions of a
space of \emph{painted} binary trees with interior edges of length in $[0,1].$ These they use to
define maps between $A_{\infty}$ spaces which preserve the multiplicative structure up to homotopy.
A later definition of the same sort of map was published by Iwase and Mimura in \cite{IM}. They
give the first detailed definition of the sequence of complexes $\cal{J}(n)$ now known as the
multiplihedra, and describe their combinatorial properties. A good review of the combinatorics of
their definition is in \cite{Kawa}. This latter reference also shows how the permuto-associahedra
can be decomposed by a combinatorial use of the multiplihedra.

 The study of the $A_{\infty}$ spaces
and their maps is still in progress. There is an open question about the correct way of defining
composition of these maps in order to form a category. In \cite{BV1} the obvious composition is
shown not to be associative. There are also interesting questions about the extension of
$A_n$-maps, as in \cite{Hemmi}, and about the transfer of $A_{\infty}$ structure through these
maps, as in \cite{Markl}. In the latter there is an open question about canonical decompositions of
the multiplihedra. The realizations we describe here lend themselves well to experimentation upon
such decompositions.

The overall structure of the associahedra is that of a topological operad, with the composition
given by inclusion. The multiplihedra together form a bimodule over this operad, with the action
again given by inclusion. This structure mirrors the fact that the spaces of painted trees form a
bimodule over the operad of spaces of trees, where  the compositions and actions are given by the
grafting of trees, root to leaf.

The multiplihedra  appear frequently in higher category theory. The definitions of  bicategory and
tricategory homomorphisms each include commuting pasting diagrams as seen in \cite{lst1} and
\cite{GPS} respectively. The two halves of the axiom for a bicategory homomorphism together form
the boundary of the multiplihedra $\cal{J}(3),$  and the two halves of the axiom for a tricategory
homomorphism together form the boundary of $\cal{J}(4).$ Since weak $n$-categories can be
understood as being the algebras of higher operads, these facts can be seen as the motivation for
defining morphisms of operad (and $n$-operad) algebras in terms of their bimodules. This definition
is mentioned in \cite{bat} and developed in detail in \cite{Hess}. In the latter paper it is
pointed out that the bimodules in question must be co-rings, which have a co-multiplication with
respect to the bimodule product over the operad.

 The multiplihedra have
appeared in many  areas related to deformation theory and $A_{\infty}$ category theory. A diagonal
map is constructed for these polytopes in \cite{umble}. This allows a functorial monoidal structure
for certain categories of $A_{\infty}$-algebras and $A_{\infty}$-categories. A different, possibly
equivalent, version of the diagonal is presented in \cite{MS}. The 3 dimensional version of the
multiplihedron is called by the name Chinese lantern diagram in \cite{Yet}, and used to describe
deformation of functors. There is a forthcoming paper by Woodward and Mau in which a new
realization of the multiplihedra as moduli spaces of disks with additional structure is presented
\cite{Mau}. This realization promises to help allow the authors and their collaborators to define
$A_n$-functors as in \cite{Wood}, as well as morphisms of cohomological field theories.


The purpose of this paper is to describe how to represent Boardman and Vogt's spaces of painted
trees with $n$ leaves as convex polytopes which are combinatorially equivalent to the CW-complexes
 described by Iwase and Mimura. Our algorithm for the vertices of the polytopes is flexible
in that it allows an initial choice of a constant $q \in (0,1)$. The boundary of the open unit
interval corresponds to certain quotient spaces of the multiplihedron. In the limit as $q\to 1$ the
convex hull approaches that of Loday's convex hull representation of the associahedra as described
in \cite{loday}. The limit as $q\to 1$ corresponds to the case for which the mapping strictly
respects the
 multiplication.

 The
limit of our algorithm as $q\to 0$ represents the case for which multiplication in the domain of
the morphism in question is strictly associative. The case for which multiplication in the range is
strictly associative was found by
 Stasheff in \cite{sta2} to yield the associahedra.  It was long assumed that the case for
which the
 domain was associative would likewise yield the associahedra, but we demonstrate in \cite{forceynew}
 that this is not
 so. In the limit as $q\to 0$ the convex hulls instead approach a newly discovered sequence of
polytopes. The low dimensional terms of this new sequence may be found in \cite{paddy} within the
axioms for pseudomonoids in a monoidal bicategory, or in \cite{Sean} within the axioms of enriched
bicategories.
 Recall that when both the range and domain are strictly associative  the multiplihedra
 become the cubes, as seen in \cite{BV1}.

The results in this paper support two related efforts of further research. The first is to describe
the important quotients of the multiplihedra just mentioned. The other project already underway is
to extend the concept of quotient multiplihedra described here to the graph associahedra
 introduced by Carr and Devadoss, in
\cite{dev}. Indeed the algorithm given here does generalize in an analogous way when applied to the
algorithm for geometric realizations of the graph associahedra invented by S. Devadoss.

In Section 2 we review the definition and properties of the multiplihedra, introducing a recursive
combinatorial definition (using the painted trees of \cite{BV1}) of the complex $\cal{J}(n)$ with
the properties described in \cite{IM}. In Section 3 we briefly give some new and provocative
combinatorial results related to the counting of the vertices of $\cal{J}(n)$. In Section 4 we
describe the method for finding geometric realizations of the multiplihedra as convex hulls. The
main result is that these convex hulls are indeed combinatorially equivalent to Stasheff's
multiplihedra. In Section 5 we relate our geometric realization to the spaces of trees defined by
Boardman and Vogt. This is done by defining a space of level trees that obeys the requirements in
\cite{BV1} and which in proof (2) of Lemma~\ref{ball} is shown directly to be homeomorphic to our
convex hull. Section 6 contains the proof of the main result by means of explicit bounding
hyperplanes for the convex hulls.

\section{Facets of the multiplihedra}
Pictures in the form of \emph{painted binary trees} can be drawn to represent the multiplication of
several objects in a monoid, before or after their passage to the image of that monoid under a
homomorphism. We use the term ``painted'' rather than ``colored'' to distinguish our trees with two
edge colorings, ``painted'' and ``unpainted,'' from the other meaning of colored, as in colored
operad or multicategory. We will refer to the exterior vertices of the tree as the  root and the
leaves , and to the interior vertices as nodes.  This will be handy since then we can reserve the
term ``vertices'' for reference to polytopes. A painted binary tree is painted beginning at the
root edge (the leaf edges are unpainted), and always painted in such a way that there are only
three types of nodes. They are:
$$
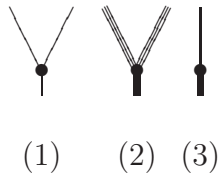
\xy 0;/r.25pc/:
  (-10,20)*=0{}="a"; (-2,20)*=0{}="b";
  (2,20)*{}="c"; (10,20)*{}="d";
  (14,20)*=0{}="e";
  (-6,12)*=0{\bullet}="v1"; (6,12)*=0{\bullet}="v2";
  (-6,4)*{\txt{\\\\(1)}}="v3"; (6,4)*{\txt{\\\\(2)}} ="v4";
  (14,12)*=0{\bullet}="v5"; (14,4)*{\txt{\\\\(3)}}="v6" \ar@{};
 "a"; "v1" \ar@{-};"a"; "v1" \ar@{-};
 "b"; "v1" \ar@{-};
 "v1"; "v3" \ar@{-}\ar@{};
 "c"; "v2" \ar@{-}; "c"; "v2" \ar@{-};
 "d"; "v2" \ar@{-};
 "v2"; "v4" \ar@{-} \ar@{};
 "c"; "v2" \ar@{=};"c"; "v2" \ar@{=};
 "d"; "v2" \ar@{=};
 "v2"; "v4" \ar@{=} \ar@{};
 "e"; "v5" \ar@{-} ;"e"; "v5" \ar@{-} \ar@{};
 "v5"; "v6" \ar@{-} ;"v5"; "v6" \ar@{-} \ar@{};
 "v5"; "v6" \ar@{=};"v5"; "v6" \ar@{=};
  \endxy
$$
This limitation on nodes implies that painted regions must be connected, that painting must never
end precisely at a trivalent node, and that painting must proceed up both branches of a trivalent
node. To see the promised representation we let the left-hand, type (1) trivalent node above stand
for multiplication in the domain; the middle, painted, type (2)  trivalent node above stand for
multiplication in the range; and the right-hand  type (3) bivalent  node stand for the action of
the mapping. For instance, given $a,b,c,d$ elements of a monoid, and $f$ a monoid morphism, the
following diagram represents the operation resulting in the product $f(ab)(f(c)f(d)).$
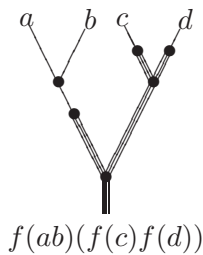
\begin{small}
 $$
\xy  0;/r.25pc/:
  (-10,20)*{a}="a"; (-2,20)*{b}="b";
  (2,20)*{c}="c"; (10,20)*{d}="d";
  (-6,12)*=0{\bullet}="v1"; (6,12)*=0{\bullet}="v2";
  (0,0)*=0{\bullet}="v3";
  (0,-7)*{\txt{\\$f(ab)(f(c)f(d))$}}="v4";
  (4,16)*=0{\bullet}="va";
  (8,16)*=0{\bullet}="vb";
  (-4,8)*=0{\bullet}="vc" \ar@{};
 "a" ;"vc" \ar@{-};"a" ;"vc" \ar@{-};
 "b" ;"v1" \ar@{-}  ;
 "c" ;"va" \ar@{-} ;
 "d" ;"vb" \ar@{-} \ar@{};
 "va";"v2" \ar@{=} ;"va";"v2" \ar@{=} ;
 "vb";"v3" \ar@{=} ;
 "vc" ;"v3" \ar@{=} ;
 "v3" ;"v4" \ar@{=} \ar@{};
 %
 %
 "va";"v2" \ar@{-} ;"va";"v2" \ar@{-} ;
 "vb";"v3" \ar@{-} ;
 "vc" ;"v3" \ar@{-} ;
 "v3" ;"v4" \ar@{-} ;
 \endxy
$$
\end{small}

Of course in the category of associative monoids and monoid homomorphisms there is no need to
distinguish the product $f(ab)(f(c)f(d))$ from $f(abcd).$ These diagrams were first introduced by
Boardman and Vogt in \cite{BV1} to help describe multiplication in (and morphisms of) topological
monoids that are not strictly associative (and whose morphisms do not strictly respect that
multiplication.)
 The $n^{th}$
multiplihedron is a  $CW$-complex whose vertices correspond to the unambiguous ways of multiplying
and applying an $A_{\infty}$-map to $n$ ordered elements of an $A_{\infty}$-space. Thus the
vertices correspond to the binary painted trees with $n$ leaves.The edges of the multiplihedra
correspond to either an association $(ab)c\to a(bc)$ or to a preservation $f(a)f(b)\to f(ab).$ The
associations can either be in the range: $(f(a)f(b))f(c)\to f(a)(f(b)f(c))$; or the image of a
domain association:  $f((ab)c)\to f(a(bc)).$

Here are the first few low dimensional multiplihedra. The  vertices are labeled, all but some of
those in the last picture. There the bold vertex in the large pentagonal facet has label
$((f(a)f(b))f(c))f(d)$ and the bold vertex in the small pentagonal facet has label $f(((ab)c)d).$
The others can be easily determined based on the fact that those two pentagons are copies of the
associahedron $\cal{K}(4),$ that is to say all their edges are associations.
$$
{\cal J}(1) = \bullet~{_{f(a)}}
  %
 %
$$
\\\\
$$
{\cal J}(2)= 
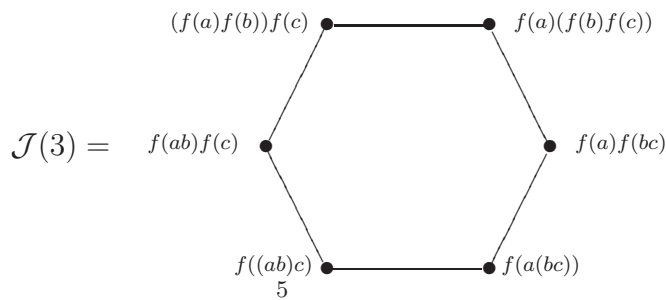
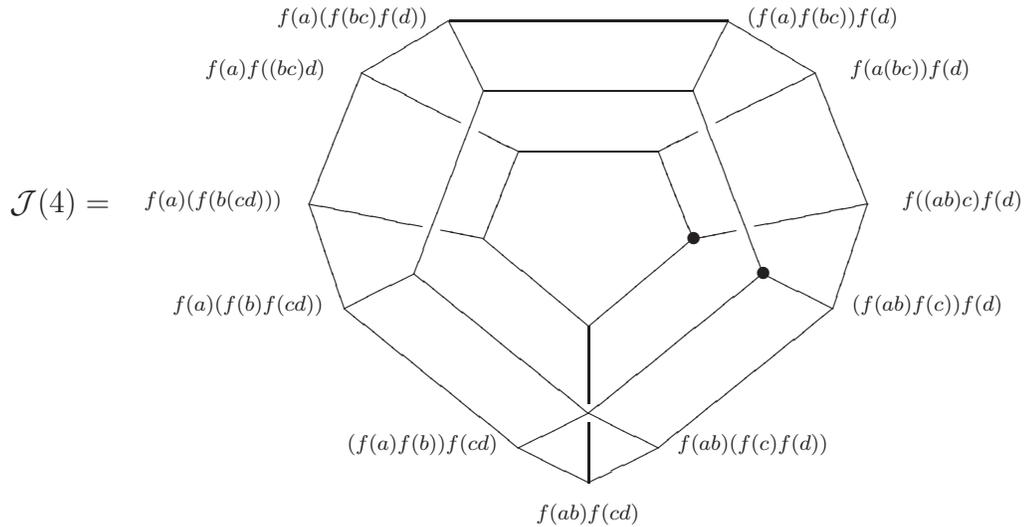
\xy 0;/r.25pc/:
    (-8,0)*{_{f(a)f(b)}~}="e";
  (8,0)*{~_{f(ab)}}="v6" \ar@{};
  "e"; "v6" \ar@{*{\bullet}-*{\bullet}} ;"e"; "v6" \ar@{-} \ar@{};
  \endxy
$$
\\\\
$$
{\cal J}(3)=\hspace{.75in} \xy 0;/r.08pc/:
 (-32,48);(0,48) *=0{^{(f(a)f(b))f(c)}\hspace{1.75in}}  \ar@{*{\bullet}-}; (-32,48);(0,48) *=0{} \ar@{-};
 (0,48);(32,48)*=0{^{f(a)(f(b)f(c))}\hspace{-1in}}  \ar@{*{\bullet}-};
  (32,48);(44,24)*=0{}  \ar@{-};
  (44,24); (56,0)*=0{^{f(a)f(bc)}\hspace{-.75in}}  \ar@{*{\bullet}-};
   (56,0); (44,-24)*=0{}  \ar@{-};
   (44,-24); (32,-48)*=0{^{f(a(bc))}\hspace{-.55in}}  \ar@{*{\bullet}-};
    (32,-48);(0,-48)*=0{}  \ar@{-};
    (0,-48); (-32,-48)*=0{^{f((ab)c)}\hspace{.55in}}  \ar@{*{\bullet}-};
     (-32,-48);(-44,-24)*=0{}  \ar@{-};
     (-44,-24); (-56,0)*=0{^{f(ab)f(c)}\hspace{.75in}}  \ar@{*{\bullet}-};
      (-56,0);(-44,24)*=0{}  \ar@{-};
      (-44,24); (-32,48)*=0{}  \ar@{};
      \endxy
$$
\\\\
$$
{\cal J}(4) = \hspace{1in}\xy 0;/r.55pc/:
   (-4,3)*=0{}="a1";
  (4,3)*=0{}="a2";
  (6,-2)*=0{\bullet}="a3";
  (0,-7)*=0{}="a4";
  (-6,-2)*=0{}="a5";
  (-6,6.5)*=0{}="b1";
  (6,6.5)*=0{}="b2";
  (10,-4)*=0{\bullet}="b3";
  (0,-12)*=0{}="b4";
  (-10,-4)*=0{}="b5";
  (-8,10.5)*=0{^{f(a)(f(bc)f(d))}\hspace{1in}}="c1";
  (8,10.5)*=0{^{(f(a)f(bc))f(d)}\hspace{-1in}}="c2";
  (13,7.5)*=0{^{f(a(bc))f(d)}\hspace{-1in}}="c3";
  (16,0)*=0{^{f((ab)c)f(d)}\hspace{-1in}}="c4";
  (14,-6)*=0{^{(f(ab)f(c))f(d)}\hspace{-1in}}="c5";
  (4,-14)*=0{^{f(ab)(f(c)f(d))}\hspace{-1in}}="c6";
  (0,-16)*=0{}="c7";
  (0,-18)*=0{^{f(ab)f(cd)}};
  (-4,-14)*=0{^{(f(a)f(b))f(cd)}\hspace{1in}}="c8";
  (-14,-6)*=0{^{f(a)(f(b)f(cd))}\hspace{1in}}="c9";
  (-16,0)*=0{^{f(a)(f(b(cd)))}\hspace{1in}}="c10";
  (-13,7.5)*=0{^{f(a)f((bc)d)}\hspace{1in}}="c11";
 (-9,-1.4)*{\hole}="x"; (-9,-1.4)*{\hole}="x";
(-6.8,4.4)*{\hole}="y";
  (9,-1.4)*{\hole}="w";
  (0,-12)*{\hole}="v";
 (6.8,4.4)*{\hole}="z";
  "a1";"a2" \ar@{-}; "a1";"a2"\ar@{-};
  "a2";"a3"  \ar@{-};
    "a3";"a4" \ar@{-};
  "a4";"a5" \ar@{-};
  "a5";"a1" \ar@{-};
 "b1";"b2" \ar@{-};
 "b2";"b3" \ar@{-};
 "b3";"b4" \ar@{-};
 "b4";"b5" \ar@{-};
 "b5";"b1" \ar@{-};
 "c1";"c2"\ar@{-};
  "c2";"c3"  \ar@{-};
    "c3";"c4" \ar@{-};
  "c4";"c5" \ar@{-};
"c5";"c6"\ar@{-};
  "c6";"c7"  \ar@{-};
    "c7";"c8" \ar@{-};
  "c8";"c9" \ar@{-};
"c9";"c10"\ar@{-};
  "c10";"c11"  \ar@{-};
  "c11";"c1"  \ar@{-};
"b1";"c1"\ar@{-}; "b2";"c2"\ar@{-}; "b3";"c5"\ar@{-};
  "b4";"c6"  \ar@{-};
  "b4";"c8"  \ar@{-};
  "b5";"c9"  \ar@{-};
"a1";"y"  \ar@{-}; "y";"c11"  \ar@{-}; "a2";"z"  \ar@{-}; "z";"c3"  \ar@{-}; "a3";"w"  \ar@{-};
"w";"c4"  \ar@{-}; "a4";"v"  \ar@{-}; "v";"c7"  \ar@{-}; "a5";"x"  \ar@{-}; "x";"c10"  \ar@{-};
     \endxy
     ~
$$
\\\\

Faces of the multiplihedra of dimension greater than zero correspond to painted trees that are no
longer binary. Here are the three new types of node allowed in a general painted tree. They
correspond to the the node types (1), (2) and (3) in that they are painted in similar fashion. They
generalize types (1), (2), and  (3) in that each has greater or equal valence than the
corresponding earlier node type.
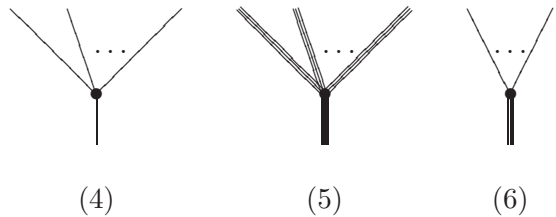
\begin{small}
 $$
\xy  0;/r.45pc/:
  (-14,20)*=0{}="a"; (-10,20)*=0{}="b";
  (-2,20)*=0{}="c"; (2,20)*=0{}="d";
  (6,20)*=0{}="e"; (14,20)*=0{}="f";
  (-7,17) *{\dots}; (9,17) *{\dots};
  (-8,14)*=0{\bullet}="v1"; (8,14)*=0{\bullet}="v2";
  (-8,8)*{\txt{\\\\(4)}}="v3";
  (8,8)*{\txt{\\\\(5)}}="v4"; \ar@{};
 (18,20)*=0{}="g"; (24,20)*=0{}="h";
 (21,17) *{\dots};
 (21,14)*=0{\bullet}="v5";
  (21,8)*{\txt{\\\\(6)}}="v6";
 "a" ;"v1" \ar@{-};"a" ;"v1" \ar@{-};
 "b" ;"v1" \ar@{-};
 "c" ;"v1" \ar@{-};
 "d" ;"v2" \ar@{-};
 "e" ;"v2" \ar@{-};
 "f" ;"v2" \ar@{-};
 "v1" ;"v3" \ar@{-};
 "v2" ;"v4" \ar@{-};
 "g" ;"v5" \ar@{-};
 "h" ;"v5" \ar@{-};
 "v5" ;"v6" \ar@{-};
 \ar@{} ;
"d" ;"v2" \ar@{=};"d" ;"v2" \ar@{=};
 "e" ;"v2" \ar@{=};
 "f" ;"v2" \ar@{=};
 "v2" ;"v4" \ar@{=};
 "v5" ;"v6" \ar@{=};
 \endxy
$$
\end{small}

\begin{definition}
By \emph{refinement} of painted trees we refer to the relationship: $t$ refines $t'$ means that
$t'$ results from the collapse of some of the internal edges of $t$. This is a partial order on
$n$-leaved painted trees, and we write $t < t'.$ Thus the binary painted trees are refinements of
the trees having nodes of type (4)-(6). \emph{Minimal refinement} refers to the following specific
case of refinement:  $t$ minimally refines $t''$ means that $t$ refines $t''$ and also that there
is no
 $t'$ such that both $t$ refines $t'$ and $t'$ refines $t''$.
\end{definition}

The recursive definition of the $n^{th}$ multiplihedron is stated by describing the type and number
of the facets, or $(n-2)$-dimensional cells. Then the boundary of  ${\cal J}(n)$ is given as the
gluing together of these facets along $(n-3)$-dimensional cells with matching associated painted
trees. Finally ${\cal J}(n)$ is defined as the cone on this boundary. It turns out that the faces
can be indexed by, or labeled by, the painted trees in such a way that the face poset of the
$n^{th}$ multiplihedron is equivalent to the face poset of the $n$-leaved painted trees. This
recasting of the definition allows the two main goals of the current paper: to unite the viewpoints
of \cite{IM} and \cite{BV1}, and to do so via a convex polytope realization.

Iwase and Mimura, however, rather than explicitly stating a recursive definition, give a geometric
definition of the $CW$-complex and then prove all the combinatorial facts about its facets. Here
(for reference sake) we reverse that order and use their theorems as definition (in terms of
painted trees).

 The type and numbers of facets of the multiplihedra are
described in \cite{IM}.


Recall that we refer to an unpainted tree with only one node as a corolla. A \emph{painted corolla}
is
 a painted tree with only one node, of type (6).
 A facet of the multiplihedron corresponds to a painted tree with only one, unpainted, interior
edge, or to a tree with all its interior edges attached to a single painted node (type (2) or (5)).

\begin{definition}\label{lfacets}
A \emph{lower tree} $l(k,s)$ is determined by a selection of $s$ consecutive leaves of the painted
corolla, $1< s \le n$, which will be the leaves of the subtree which has the sole interior edge as
its root edge.
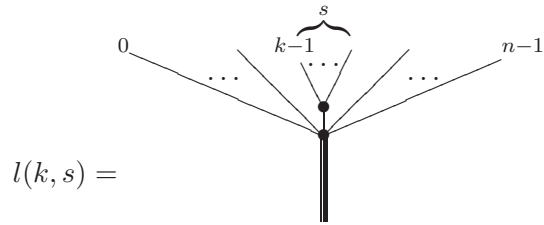
\begin{small}
 $$\xy (0,0) *{ l(k,s) =} \endxy
\xy (0,0) *{  \xy  0;/r.45pc/:
    (-8,22) *{{s \atop \overbrace{~~~~~~~~~}}};
   (-22,20)*{{}^0}="L";(6,20)*{{}^{n-1}}="R";
  (-14,20)*=0{}="a"; (-10,20)*{{}^{k-1}}="b";
  (-6,20)*=0{}="b2";
  (-2,20)*=0{}="c";
  (-8,14)*=0{\bullet}="v1";
  (-8,8)*=0{}="v3";
  (-8,16)*=0{\bullet}="v0";
  (-15,18) *{\dots};
  (-1,18) *{\dots};
  (-8,19) *{\dots};
   "a" ;"v1" \ar@{-};"a" ;"v1" \ar@{-};
   "L" ;"v1" \ar@{-};
   "R" ;"v1" \ar@{-};
 "b" ; "v0" \ar@{-};
 "b2" ; "v0" \ar@{-};
 "v0"; "v1" \ar@{-};
 "v1" ; "c" \ar@{-};
  "v1" ;"v3" \ar@{-};
  \ar@{} ;
"v1" ;"v3" \ar@{=};"v1" ;"v3" \ar@{=};
 \endxy } \endxy
$$
\end{small}

 To each lower tree corresponds a   \emph{lower facet} of the multiplihedron, which in \cite{IM} are denoted
  ${\cal J}_k(r,s)$ where $r = n+1-s.$ Here $k$ is the first
``gap between branches'' of the $s-1$ consecutive gaps (that is, $k-1$ is the first leaf of the $s$
consecutive leaves.)  In the complex ${\cal J}(n)$ defined in \cite{IM} the lower facet  ${\cal
J}_k(r,s)$ is a combinatorial copy of the complex ${\cal J}(r) \times{\cal K}(s).$
\end{definition}
\begin{definition}\label{ufacets}
 The \emph{upper trees} $u(t;r_1,\dots, r_t)$ with all interior (necessarily painted) edges attached to a single painted node will
 appear thus:
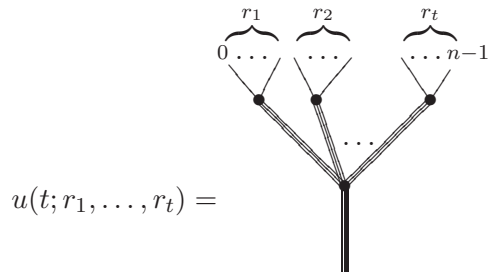
\begin{small}
 $$\xy (0,0) *{ u(t;r_1,\dots, r_t) =} \endxy
\xy (0,0) *{  \xy  0;/r.45pc/:
  (-.5,23)*{{}^0}="a"; (1.5,25.25) *{{r_1 \atop \overbrace{~~~}}}; (1.5,23) *{\dots}; (3.5,23)*=0{}="b";
(4.5,23)*=0{}="a2"; (6.5,25.25) *{{r_2 \atop \overbrace{~~~}}}; (6.5,23)  *{\dots};
(8.5,23)*=0{}="b2";
 (12,23)*=0{}="a3"; (14,25.25) *{{r_t \atop \overbrace{~~~}}}; (14,23) *{\dots~~}; (16,23)*{~~^{~n-1}}="b3";
   (2,20)*=0{\bullet}="d";
  (6,20)*=0{\bullet}="e"; (14,20)*=0{\bullet}="f";
 (9,17) *{\dots};
  (8,14)*=0{\bullet}="v2";
  (8,8)*=0{}="v4"; \ar@{};
 "d" ;"v2" \ar@{-}; "d" ;"v2" \ar@{-};
"a" ;"d" \ar@{-};
 "b" ;"d" \ar@{-};
  "a2" ;"e" \ar@{-};
   "b2" ;"e" \ar@{-};
 "a3" ;"f" \ar@{-};
 "b3" ;"f" \ar@{-};
 "e" ;"v2" \ar@{-};
 "f" ;"v2" \ar@{-};
  "v2" ;"v4" \ar@{-};
 \ar@{} ;
"d" ;"v2" \ar@{=};"d" ;"v2" \ar@{=};
 "e" ;"v2" \ar@{=};
 "f" ;"v2" \ar@{=};
 "v2" ;"v4" \ar@{=};
 \endxy } \endxy
$$
\end{small}

  In \cite{IM} the corresponding \emph{upper facets} are labeled ${\cal J}(t;r_1,\dots ,r_t).$ Here  $t$ is the number
   of painted interior edges
   and $r_i$ is the number of
  leaves in the subtree supported by the $i^{th} $ interior edge.
  In the complex
${\cal J}(n)$ defined in \cite{IM} the upper facet  ${\cal J}(t;r_1,\dots ,r_t)$ is a combinatorial
copy of the complex  ${\cal K}(t)\times{\cal J}(r_1)\times\dots\times{\cal J}(r_t).$
\end{definition}
Here is a quick count of upper and lower facets, agreeing precisely with that given in \cite{IM}.
\begin{theorem}\cite{IM}
The number of facets of the $n^{th}$ multiplihedron is:
$$
\frac{n(n-1)}{2}+2^{(n-1)}-1.
$$
\end{theorem}
\begin{proof}
The number of lower trees is $\displaystyle{\frac{n(n-1)}{2}}.$  This follows easily from summing
the ways of choosing $s-1$ consecutive ``gaps between branches'' of the corolla, corresponding to
the choice of $s$ consecutive leaves.   Note that this count includes one more than the count of
the facets of the associahedron, since it includes the possibility of selecting all $n$ leaves.

The upper trees are determined by choosing any size $k$ proper subset of the ``spaces between
branches''
 of the painted corolla, $1 \le k < n-1 $.
 Each set of consecutive
 ``spaces between branches'' in that list of $k$ chosen spaces determines a set of consecutive leaves which
 will be the leaves of a subtree (that is itself a painted corolla)
  with its root edge one of the painted interior edges. If neither of the adjacent spaces to a
  given branch are chosen, its leaf will be the sole leaf of a subtree that is a painted corolla
  with only one leaf.
  Thus we count upper trees by $\displaystyle{ \sum_{k=0}^{n-2} {n-1\choose k} = 2^{(n-1)}-1 }.$
\end{proof}

The construction of the $n^{th}$ multiplihedron may be inductively accomplished by collecting its
facets, and then labeling their faces. The following definition is identical to the properties
demonstrated in \cite{IM}.
\begin{definition}
The first multiplihedron denoted $\cal{J}(1)$ is defined to be the single point $\{*\}.$ It is
associated to the painted tree with one leaf, and thus one type (3) internal node. Assume that the
$\cal{J}(k)$ have been defined for $k=1\dots n-1.$ To $\cal{J}(k)$ we associate the $k$-leaved
painted corolla. We define an $(n-2)$-dimensional
 $CW$-complex $\partial\cal{J}(n)$
as follows, and then define $\cal{J}(n)$ to be the cone on $\partial\cal{J}(n)$.  Now the
top-dimensional cells of $\partial\cal{J}(n)$ (upper and lower facets of $\cal{J}(n)$) are in
 bijection with the set of  painted trees of two types, upper and lower trees as defined above.

 Each sub-facet of an upper or lower facet is labeled with a tree
that is a refinement of the upper or lower tree. Since the facets are products, their sub-facets in
turn are products of faces (of smaller associahedra and multiplihedra) whose dimensions sum to
$n-3.$ Each of these sub-facets thus comes (inductively) with a list of associated trees. There
will always be a unique way of grafting the trees on this list to construct a painted tree that is
a minimal refinement of the upper or
 lower tree associated to the facet in question.
For the sub-facets of an upper facet the recipe is to paint entirely the $t$-leaved tree associated
to a face of $\cal{K}(t)$ and to graft to each of its branches in turn the trees associated to the
appropriate faces of  $\cal{J}(r_1)$ through $\cal{J}(r_t)$ respectively.
 A sub-facet of the
 lower facet $\cal{J}_k(r,s)$ inductively comes with pair of trees.
 The recipe
 for assigning our sub-facet
  an $n$-leaved
 minimal refinement of the $n$-leaved minimal lower tree $l(k,s)$ is to graft the unpainted $s$-leaved tree to the
 $k^{th}$
 leaf of the painted $r$-leaved tree.

 The intersection of two facets in the boundary of $\cal{J}(n)$ occurs along sub-facets of each which
have associated painted trees that are identical. Then $\cal{J}(n)$ is defined to be the cone on
$\partial\cal{J}(n).$
 To $\cal{J}(n)$ we assign the painted corolla of $n$ leaves.
\end{definition}
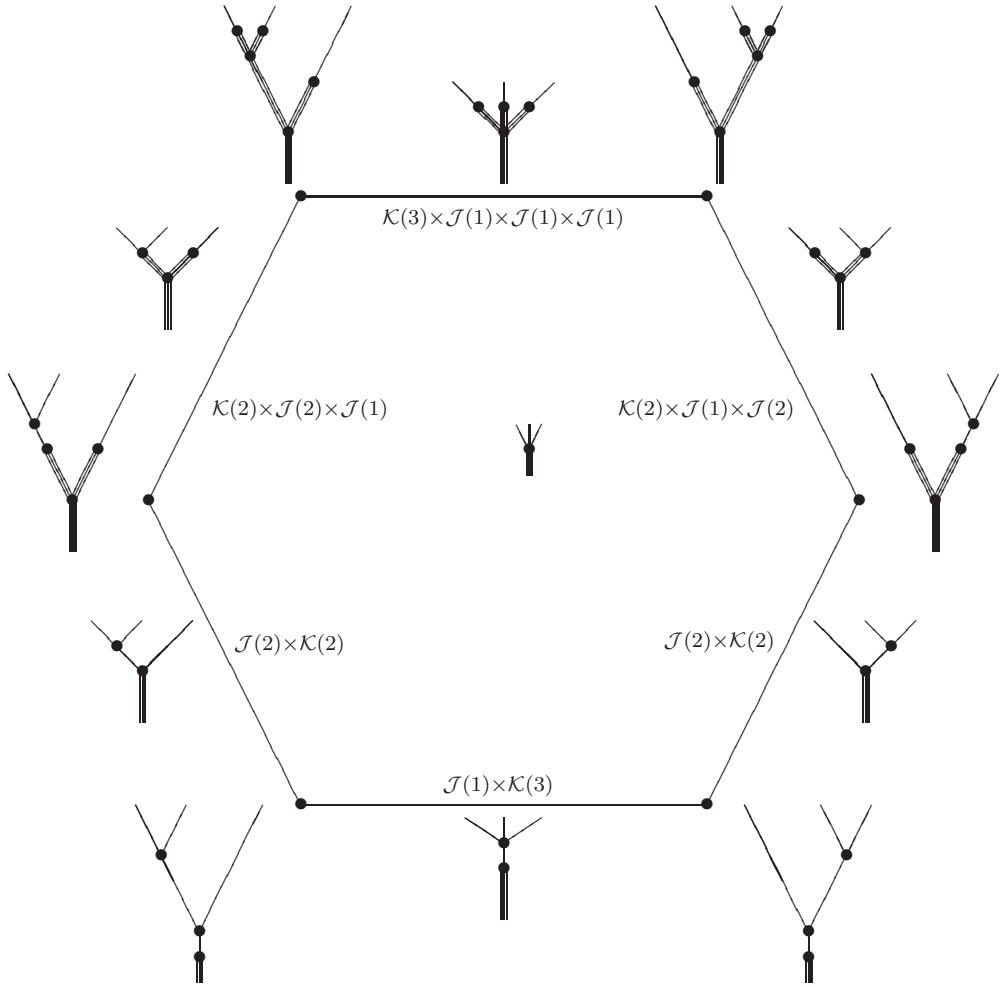
\begin{remark}
The listing of types and enumeration of facets above corresponds to properties (2-a) through (2-c)
of \cite{IM}. The intersection of facets described in the definition corresponds to properties
(c-1) through (c-4) in \cite{IM}.
\end{remark}
\begin{example}
$$
\cal{J}(1) = \bullet~\hspace{.05in} \xy 0;/r.25pc/:
    (14,6)*=0{}="e";
  (14,2)*=0{\bullet}="v5"; (14,-3)*=0{}="v6" \ar@{};
  "e"; "v5" \ar@{-} ;"e"; "v5" \ar@{-} \ar@{};
 "v5"; "v6" \ar@{-} ;"v5"; "v6" \ar@{-} \ar@{};
 "v5"; "v6" \ar@{=};"v5"; "v6" \ar@{=};
  \endxy
$$

Here is the complex $\cal{J}(2)$ with the upper facet $\cal{K}(2) \times \cal{J}(1)\times
\cal{J}(1)$ on the left and the lower facet $\cal{J}(1)\times \cal{K}(2)$ on the right:

\begin{small}
$$
\xy 0;/r.55pc/: (-12,8)*=0{}="a"; (-4,8)*=0{}="b";
  (4,8)*=0{}="e"; (12,8)*=0{}="f";
  (-2,4)*=0{}="c"; (2,4)*=0{}="d";
(-10,6)*=0{\bullet}="va";
 (-6,6)*=0{\bullet}="vb";
(-8,4)*=0{\bullet}="v1";
 (0,2)*=0{\bullet}="v2";
  (8,4)*=0{\bullet}="v3";
  (8,2)*=0{\bullet}="v4";
  (-8,0)*=0{}="r1";
(0,0)*=0{}="r2";
 (8,0)*=0{}="r3";
 (-8,-1)*=0{\bullet}="lp";
 (0,-1)*=0{}="cp";
 (8,-1)*=0{\bullet}="rp";
"a" ;"v1" \ar@{-}; "a" ;"v1" \ar@{-};
 "b" ;"v1" \ar@{-};
 "v1" ;"r1" \ar@{-};
 "c" ;"v2" \ar@{-};
 "d" ;"v2" \ar@{-};
 "v2" ;"r2" \ar@{-};
 "e" ;"v3" \ar@{-};
 "f" ;"v3" \ar@{-};
 "v3" ;"v4" \ar@{-};
 "v4" ;"r3" \ar@{-};
 "lp" ;"rp" \ar@{-};
 "va" ;"v1" \ar@{=}; "va" ;"v1" \ar@{=};
 "vb" ;"v1" \ar@{=};
 "v1" ;"r1" \ar@{=};
 "v2" ;"r2" \ar@{=};
 "v4" ;"r3" \ar@{=};
 \endxy
$$
\end{small}

And here is the complex $\cal{J}(3).$ The product structure of facets is listed.  Notice  how the
sub-facets (vertices) are labeled. For instance, the upper right vertex is labeled by a tree that
could be constructed by grafting three copies of the single leaf painted corolla onto a completely
painted binary tree with three leaves, or by grafting a single leaf painted corolla and a 2-leaf
painted binary tree onto the leaves of a 2-leaf (completely) painted binary tree.

\begin{small}
$$
\xy 0;/r.2pc/:
 (-32,48);(0,48) *=0{}  \ar@{-}; (-32,48);(0,48) *=0{}  \ar@{-};
 (0,44) *={\txt{$^{\cal{K}(3)\times\cal{J}(1)\times\cal{J}(1)\times\cal{J}(1)}$}};
 (0,48);(32,48)*=0{\bullet}  \ar@{-};
   (32,48);(44,24)*=0{}  \ar@{-};
   (32,14) *={^{\cal{K}(2)\times\cal{J}(1)\times\cal{J}(2)}};
   (-32,14) *={^{\cal{K}(2)\times\cal{J}(2)\times\cal{J}(1)}};
  (44,24); (56,0)*=0{\bullet}  \ar@{-};
   (56,0); (44,-24)*=0{}  \ar@{-} ^>>>>>>>>>>>>>>>>>>>>>{\cal{J}(2)\times\cal{K}(2)};
   (44,-24); (32,-48)*=0{\bullet}  \ar@{-};
    (32,-48);(0,-48)*=0{}  \ar@{-} _>>>>>>>>>>>>>>>>>>>>>>>>{\cal{J}(1)\times\cal{K}(3)};;
    (0,-48); (-32,-48)*=0{\bullet}  \ar@{-};
     (-32,-48);(-44,-24)*=0{}  \ar@{-} ^>>>>>>>>>>>>>>>>>>>>>{\cal{J}(2)\times\cal{K}(2)};
     (-44,-24); (-56,0)*=0{\bullet} \ar@{-};
      (-56,0);(-44,24)*=0{}  \ar@{-};;
      (-44,24); (-32,48)*=0{\bullet}  \ar@{};
            (4,4);(4,8) \ar@{=}; (4,4);(4,8) *=0{\bullet}\ar@{-}; (4,4);(4,8) \ar@{-};
      (4,8);(2,12)  \ar@{-}; (4,8); (4,12) \ar@{-}; (4,8); (6,12) \ar@{-};
(   -44 ,   78  )      *=0{}="aa";
 (   -36 ,   78  )      *=0{}="ab";
  (   -32 ,   78  ) *=0{}="ac";
 (   -24 ,   78  )      *=0{}="ad";
  (   -40 ,   70  )      *=0{\bullet}="av1";
   (   -28 ,70  )   *=0{}="av2";
 (   -34 ,   58  )      *=0{\bullet}="av3";
  (   -34 ,   50  )      *=0{}="av4";
(   -30 ,   74  )      *=0{}="ava";
 (   -26 ,   74  )      *=0{}="avb";
  (   -38 ,   74  )*=0{\bullet}="ave";
 (   -42 ,   74  )      *=0{\bullet}="avf";
  (   -38 ,   66  )      *=0{}="avc";
(   -30 ,   66  )   *=0{\bullet}="avd";
 "aa" ;"avc" \ar@{-};"aa" ;"avc" \ar@{-};
 "ab" ;"av1" \ar@{-}  ;
  "ad" ;"avb" \ar@{-} \ar@{};
 "ave";"av1" \ar@{=};"ave";"av1" \ar@{=} ;
 "avf";"av1" \ar@{=} ;
 "av1" ;"av3" \ar@{=} ;
 "avd" ;"av3" \ar@{=} ;
 "av3" ;"av4" \ar@{=} \ar@{};
"avb";"av3" \ar@{-} ;
 "avb";"av3" \ar@{-} ;
 "avc" ;"av3" \ar@{-} ;
 "av3" ;"av4" \ar@{-} ;
 (   24  ,   78  )   *=0{}="ba";
(   32  ,   78  )   *=0{}="bb";
 (   36  ,   78  )   *=0{}="bc";
  (   44  ,   78  )   *=0{}="bd";
   (28  ,   70  )   *=0{}="bv1";
 (   40  ,   70  )   *=0{\bullet}="bv2";
  (   34  ,   58  )*=0{\bullet}="bv3";
 (   34  ,   50  )   *=0{}="bv4";
  (   38  ,   74  )   *=0{\bullet}="bva";
   (   42,   74  )   *=0{\bullet}="bvb";
 (   30  ,   74  )   *=0{}="bve";
  (   26  ,   74  )   *=0{}="bvf";
   (30  ,   66  )   *=0{\bullet}="bvc" ;
 (   38  ,   66  )   *=0{}="bvd" ;
 "ba" ;"bvc" \ar@{-};"ba" ;"bvc" \ar@{-};
  "bc" ;"bva" \ar@{-} ;
 "bd" ;"bvb" \ar@{-} \ar@{};
 "bva";"bv2" \ar@{=} ;"bva";"bv2" \ar@{=} ;
 "bvb";"bv2" \ar@{=} ;
 "bv2";"bv3" \ar@{=} ;
 "bvc" ;"bv3" \ar@{=} ;
 "bv3" ;"bv4" \ar@{=} \ar@{};
 "bva";"bv2" \ar@{-} ;"bva";"bv2" \ar@{-} ;
 "bvb";"bv3" \ar@{-} ;
 "bvc" ;"bv3" \ar@{-} ;
 "bv3" ;"bv4" \ar@{-} ;
(   -78 ,   20  )            *=0{}="ca";
 (   -70 ,   20  )            *=0{}="cb";
  (   -66 ,   20  )*=0{}="cc";
 (   -58 ,   20  )            *=0{}="cd";
  (   -74 ,   12  )*=0{\bullet}="cv1";
 (   -62 ,   12  )            *=0{}="cv2";
  (   -68 ,   0   )*=0{\bullet}="cv3";
 (   -68 ,   -8  )   *=0{}="cv4";
  (   -64 ,   16  )            *=0{}="cva";
   (-60 ,   16  )            *=0{}="cvb";
 (   -72 ,   16  )      *=0{}="cve";
  (   -76 ,   16  )*=0{}="cvf";
 (   -72 ,   8   )   *=0{\bullet}="cvc" ;
  (   -64 ,   8   )*=0{\bullet}="cvd" ;
 "ca" ;"cv1" \ar@{-};"ca" ;"cv1" \ar@{-};
 "cv1" ;"cvc" \ar@{-};
 "cb" ;"cv1" \ar@{-}  ;
  "cd" ;"cvb" \ar@{-} \ar@{};
 "cvc" ;"cv3" \ar@{=} ;
 "cvc" ;"cv3" \ar@{=} ;
 "cvd" ;"cv3" \ar@{=} ;
 "cv3" ;"cv4" \ar@{=} \ar@{};
"cvb";"cv3" \ar@{-} ;
 "cvb";"cv3" \ar@{-} ;
 "cvc" ;"cv3" \ar@{-} ;
 "cv3" ;"cv4" \ar@{-} ;
(   58  ,   20  )      *=0{}="da";
 (   66  ,   20  )      *=0{}="db";
  (   70  ,   20  )*=0{}="dc";
 (   78  ,   20  )      *=0{}="dd";
  (   62  ,   12  )      *=0{}="dv1";
   (   74  ,   12)   *=0{\bullet}="dv2";
 (   68  ,   0   )      *=0{\bullet}="dv3";
  (   68  ,   -8  )*=0{}="dv4";
 (   72  ,   16  )      *=0{}="dva";
  (   76  ,   16  )      *=0{}="dvb";
   (   64  ,   16)      *=0{}="dve";
 (   60  ,   16  )      *=0{}="dvf";
  (   64  ,   8   )      *=0{\bullet}="dvc" ;
(   72  ,   8   )   *=0{\bullet}="dvd" ;
 "da" ;"dvc" \ar@{-};"da" ;"dvc" \ar@{-};
  "dc" ;"dva" \ar@{-} ;
 "dd" ;"dvb" \ar@{-} \ar@{};
"dvd" ;"dv3" \ar@{=} ;"dvd" ;"dv3" \ar@{=} ;
 "dvc" ;"dv3" \ar@{=} ;
 "dv3" ;"dv4" \ar@{=} \ar@{};
 "dva";"dv2" \ar@{-} ;"dva";"dv2" \ar@{-} ;
 "dvb";"dv2" \ar@{-} ;
 "dv2";"dvd" \ar@{-} ;
  "dvd";"dv3" \ar@{-} ;
 "dvc" ;"dv3" \ar@{-} ;
 "dv3" ;"dv4" \ar@{-} ;
(   -58 ,   -48 )   *=0{}="ea";
 (   -50 ,   -48 )   *=0{}="eb";
  (   -46 ,   -48 )   *=0{}="ec";
   (-38 ,   -48 )   *=0{}="ed";
 (   -54 ,   -56 )   *=0{\bullet}="ev1";
  (   -42 ,   -56 )*=0{}="ev2";
 (   -48 ,   -68 )   *=0{\bullet}="ev3";
  (   -48 ,   -76 )   *=0{}="ev4";
   (   -48 ,-72 )   *=0{\bullet}="ev5";
 (   -44 ,   -52 )   *=0{}="eva";
  (   -40 ,   -52 )   *=0{}="evb";
   (-52 ,   -52 )   *=0{}="eve";
 (   -56 ,   -52 )   *=0{}="evf";
  (   -52 ,   -60 )   *=0{}="evc" ;
   (-44 ,   -60 )   *=0{}="evd" ;
 "ea" ;"evc" \ar@{-};"ea" ;"evc" \ar@{-};
 "eb" ;"ev1" \ar@{-}  ;
  "ed" ;"evb" \ar@{-} \ar@{};
   "ev5" ;"ev4" \ar@{=}; "ev5" ;"ev4" \ar@{=} \ar@{};
"evb";"ev3" \ar@{-} ;
 "evb";"ev3" \ar@{-} ;
 "evc" ;"ev1" \ar@{-} ;
 "ev1" ;"ev3" \ar@{-} ;
 "ev3" ;"ev5" \ar@{-} ;
 "ev5" ;"ev4" \ar@{-} ;
  (   38  ,   -48 )   *=0{}="fa";
(   46  ,   -48 )   *=0{}="fb";
 (   50  ,   -48 )   *=0{}="fc";
  (   58  ,   -48 )   *=0{}="fd";
   (42  ,   -56 )   *=0{}="fv1";
 (   54  ,   -56 )   *=0{\bullet}="fv2";
  (   48  ,   -68 )*=0{\bullet}="fv3";
 (   48  ,   -76 )   *=0{}="fv4";
  (   48  ,   -72 )   *=0{\bullet}="fv5";
   (   52,   -52 )   *=0{}="fva";
 (   56  ,   -52 )   *=0{}="fvb";
  (   44  ,   -52 )   *=0{}="fve";
   (   40,   -52 )   *=0{}="fvf";
 (   44  ,   -60 )   *=0{}="fvc" ;
  (   52  ,   -60 )   *=0{}="fvd" ;
 "fa" ;"fvc" \ar@{-};"fa" ;"fvc" \ar@{-};
  "fc" ;"fva" \ar@{-} ;
 "fd" ;"fvb" \ar@{-} \ar@{};
  "fv5" ;"fv4" \ar@{=}; "fv5" ;"fv4" \ar@{=} \ar@{};
 "fva";"fv2" \ar@{-} ;"fva";"fv2" \ar@{-} ;
  "fvb";"fv2" \ar@{-} ;
 "fv2";"fv3" \ar@{-} ;
 "fvc" ;"fv1" \ar@{-} ;
 "fv1" ;"fv3" \ar@{-} ;
 "fv3" ;"fv5" \ar@{-} ;
 "fv5" ;"fv4" \ar@{-} ;
 (-8,66)*=0{}="a"; (8,66)*=0{}="b";
 (0,66)*=0{}="c";
  (-4,62)*=0{\bullet}="va";
  (0,62)*=0{\bullet}="vc";
 (4,62)*=0{\bullet}="vb";
(0,58)*=0{\bullet}="v1";
  (0,50)*=0{}="r1";
"a" ;"v1" \ar@{-}; "a" ;"v1" \ar@{-};
  "c" ; "v1" \ar@{-};
 "b" ;"v1" \ar@{-};
 "v1" ;"r1" \ar@{-};
 "va" ;"v1" \ar@{=}; "va" ;"v1" \ar@{=};
 "vb" ;"v1" \ar@{=};
 "vc" ; "v1" \ar@{=};
 "v1" ;"r1" \ar@{-};
(45,43)*=0{}="a"; (61,43)*=0{}="b";
 (53,43)*=0{}="c";
  (49,39)*=0{\bullet}="va";
 (57,39)*=0{\bullet}="vb";
(53,35)*=0{\bullet}="v1";
  (53,27)*=0{}="r1";
"a" ;"v1" \ar@{-}; "a" ;"v1" \ar@{-};
  "c" ; "vb" \ar@{-};
 "b" ;"v1" \ar@{-};
 "v1" ;"r1" \ar@{-};
 "va" ;"v1" \ar@{=}; "va" ;"v1" \ar@{=};
 "vb" ;"v1" \ar@{=};
 "v1" ;"r1" \ar@{-};
(-45,43)*=0{}="a";
 (-61,43)*=0{}="b";
 (-53,43)*=0{}="c";
  (-49,39)*=0{\bullet}="va";
 (-57,39)*=0{\bullet}="vb";
(-53,35)*=0{\bullet}="v1";
  (-53,27)*=0{}="r1";
"a" ;"v1" \ar@{-}; "a" ;"v1" \ar@{-};
  "c" ; "vb" \ar@{-};
 "b" ;"v1" \ar@{-};
 "v1" ;"r1" \ar@{-};
 "va" ;"v1" \ar@{=}; "va" ;"v1" \ar@{=};
 "vb" ;"v1" \ar@{=};
 "v1" ;"r1" \ar@{-};
 (-6,-50)*=0{}="a"; (6,-50)*=0{}="b";
 (0,-50)*=0{}="c";
    (0,-54)*=0{\bullet}="vc";
 (0,-58)*=0{\bullet}="v1";
  (0,-66)*=0{}="r1";
"a" ;"vc" \ar@{-}; "a" ;"vc" \ar@{-};
  "c" ; "vc" \ar@{-};
 "b" ;"vc" \ar@{-};
 "v1" ;"r1" \ar@{-};
  \ar@{-};
 "vc" ; "v1" \ar@{=};
 "v1" ;"r1" \ar@{-};
(49,-19)*=0{}="a"; (65,-19)*=0{}="b";
 (57,-19)*=0{}="c";
  (53,-23)*=0{}="va";
   (61,-23)*=0{\bullet}="vb";
(57,-27)*=0{\bullet}="v1";
  (57,-35)*=0{}="r1";
"a" ;"v1" \ar@{-}; "a" ;"v1" \ar@{-};
  "c" ; "vb" \ar@{-};
 "b" ;"v1" \ar@{-};
 "v1" ;"r1" \ar@{-};
 "va" ;"v1" \ar@{-}; "va" ;"v1" \ar@{-};
 "vb" ;"v1" \ar@{=};
  "v1" ;"r1" \ar@{-};
(-49,-19)*=0{}="a"; (-65,-19)*=0{}="b";
 (-57,-19)*=0{}="c";
  (-53,-23)*=0{}="va";
   (-61,-23)*=0{\bullet}="vb";
(-57,-27)*=0{\bullet}="v1";
  (-57,-35)*=0{}="r1";
"a" ;"v1" \ar@{-}; "a" ;"v1" \ar@{-};
  "c" ; "vb" \ar@{-};
 "b" ;"v1" \ar@{-};
 "v1" ;"r1" \ar@{-};
 "va" ;"v1" \ar@{-}; "va" ;"v1" \ar@{-};
 "vb" ;"v1" \ar@{=};
  "v1" ;"r1" \ar@{-};
 \endxy
$$
\end{small}
\end{example}
\section{Vertex Combinatorics}
Now for a new result about the counting of the binary painted trees with $n$ leaves.
\begin{theorem}
 The number of vertices $a_n$ of the $n^{th}$ multiplihedron is given recursively by:
$$
a_n = C(n-1) +  \sum_{i=1}^{n-1}a_ia_{n-i}
$$
where $a_0 = 0$ and $C(n-1)$ are the Catalan numbers, which count binary (unpainted) trees as well
as the vertices of the associahedron.
\end{theorem}
\begin{proof}
The Catalan numbers $C(n-1)$ count those vertices which correspond to the painted binary trees with
$n$ leaves which have only the root painted, that is only nodes of type (1) and (3). Now we count
the trees for which the initial (lowest) trivalent node is painted (type (2)). Each of these
consists of a choice of two painted binary subtrees whose root is the initial painted node, and
whose leaves must sum to $n.$ Thus we sum over the ways that $n$ can be split into two natural
numbers.
\end{proof}
\begin{remark}
This formula gives the sequence which begins: $$0,1,2,6,21,80,322,1348,5814\dots.$$ It is sequence
A121988 of the On-line Encyclopedia of integer sequences. The recursive formula  above yields the
equation
$$A(x) = xc(x)+ (A(x))^2$$
where $A(x)$ is the ordinary generating function of the sequence $a_n$ above and $c(x)$ is the
generating function for the Catalan numbers $C(n).$ (So $xc(x)$ is the generating function for the
sequence $\{C(n-1)\}_{n=0}^{\infty}.$) Recall that $c(x) = \frac{1-\sqrt{1-4x}}{2x}.$ Thus by use
of the quadratic formula we have
 $$A(x) = \frac{1-\sqrt{2\sqrt{1-4x}-1}}{2}.$$

It is not hard to check that therefore
 $A(x) = xc(x)c(xc(x)).$
 The Catalan transform of a sequence $b_n$ with generating function $B(x)$ is defined in
 \cite{Barry} as the sequence with generating function $B(xc(x)).$ Since $xc(x)$ is the
 generating function of $C(n-1)$ then the number of vertices of the $n^{th}$ multiplihedron
 is given by the Catalan transform of the  Catalan numbers $C(n-1).$  Thus the theorems of
 \cite{Barry} apply, for instance: a formula for the number of vertices is given by
 $$a_n = \frac{1}{n}\sum_{k=1}^n{2n-k-1 \choose n-1}{2k-2 \choose k-1}~;~ a_0 = 0.$$

 We note that $A(x) = B(x)c(B(x))$ for $B(x) = xc(x).$ It may be that taking a generating function $B(x)$ to
 the new one given by
 $B(x)c(B(x))$ is the definition of
 a new kind of
 Catalan transform that would be interesting to study in its own right.
\end{remark}

\section{An algorithm for the extremal points}

In \cite{loday} Loday gives an algorithm for taking the binary trees with $n$ leaves and finding
for each an extremal point in {\bf R}$^{n-1}$; together whose convex hull is ${\cal K}(n),$ the
$(n-2)$-dimensional associahedron. Note that Loday writes formulas with the convention that the
number of leaves is $n+1,$ where we instead always use $n$ to refer to the number of leaves. Given
a (non-painted) binary $n$-leaved tree $t,$ Loday arrives at a point $M(t)$ in {\bf R}$^{n-1}$ by
calculating a coordinate from each trivalent node. These are ordered left to right based upon the
ordering of the leaves from left to right. Following Loday we number the leaves $0,1,\dots,n-1 $
and the nodes $1,2,\dots,n-1.$ The $i^{th}$ node is ``between'' leaf $i-1$ and leaf $i$ where
``between'' might be described to mean that a rain drop falling between those leaves would be
caught at that node. Each trivalent node has a left and right branch, which each support a subtree.
To find the Loday coordinate for the  $i^{th}$ node we take the product of the number of leaves of
the left subtree ($l_i$) and the number of leaves of the right subtree ($r_i$) for that node. Thus
$M(t) = (x_1, \dots x_{n-1})$ where $x_i = l_ir_i$. Loday proves that the convex hull of the points
thus calculated for all $n$-leaved binary trees  is the $n^{th}$ associahedron. He also shows that
the points thus calculated all lie in the $n-2$ dimensional affine hyperplane $H$ given by the
equation $x_1+\dots+x_{n-1} = S(n-1) = {1\over 2}n(n-1).$

We adjust Loday's algorithm to apply to painted binary trees as described above, with only nodes of
type (1), (2), and (3), by choosing a number $q \in (0,1).$ Then given a painted binary tree $t$
with $n$ leaves we calculate a point $M_q(t)$ in {\bf R}$^{n-1}$ as follows: we begin by finding
the coordinate for each trivalent node from left to right given by Loday's algorithm, but if the
node is of type (1) (unpainted, or colored by the domain) then its new coordinate is found by
further multiplying its Loday coordinate by $q$.  Thus
$$M_q(t) = (x_1, \dots x_{n-1}) \text{ where } x_i = \begin{cases}
    ql_ir_i,& \text{if node $i$ is type (1)} \\
    l_ir_i,& \text{if node $i$ is type (2).}
\end{cases} $$
Note that whenever we speak of the numbered nodes ($1,\dots, n-1$  from left to right) of a binary
tree, we are referring only to the trivalent nodes, of type (1) or (2). For an example, let us
calculate the point in {\bf R}$^3$ which corresponds to the 4-leaved tree:

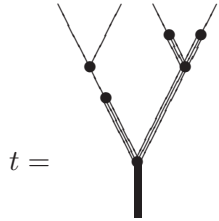
\begin{small}
 $$ t =
\xy  0;/r.25pc/:
  (-10,20)*=0{}="a"; (-2,20)*=0{}="b";
  (2,20)*=0{}="c"; (10,20)*=0{}="d";
  (-6,12)*=0{\bullet}="v1"; (6,12)*=0{\bullet}="v2";
  (0,0)*=0{\bullet}="v3";
  (0,-7)*=0{}="v4";
  (4,16)*=0{\bullet}="va";
  (8,16)*=0{\bullet}="vb";
  (-4,8)*=0{\bullet}="vc" \ar@{};
 "a" ;"v1" \ar@{-};"a" ;"v1" \ar@{-};
 "v1" ;"vc" \ar@{-};
 "b" ;"v1" \ar@{-}  ;
 "c" ;"va" \ar@{-} ;
 "d" ;"vb" \ar@{-} \ar@{};
 "va";"v2" \ar@{=} ;"va";"v2" \ar@{=} ;
 "vb";"v2" \ar@{=} ;
 "v2" ;"v3" \ar@{=};
 "vc" ;"v3" \ar@{=} ;
 "v3" ;"v4" \ar@{=} \ar@{};
 %
 %
 "va";"v2" \ar@{-} ;"va";"v2" \ar@{-} ;
 "vb";"v3" \ar@{-} ;
 "vc" ;"v3" \ar@{-} ;
 "v3" ;"v4" \ar@{-} ;
 \endxy
$$
\end{small}

Now $M_q(t) = (q, 4, 1). $
\begin{theorem}\label{main} The convex hull of all the resulting points $M_q(t)$ for $t$ in the set of $n$-leaved binary
painted trees is the $n^{th}$ multiplihedron. That is, our convex hull is combinatorially
equivalent to the CW-complex ${\cal J}(n)$ defined by Iwase and Mimura, and is homeomorphic to the
space of level (painted) trees defined by Boardman and Vogt.
\end{theorem}
The proof will follow in section~\ref{proofsec}.
\begin{example}
\end{example}
Here are all the painted binary trees with 3 leaves, together with their points $M_q(t)\in${\bf
R}$^2.$
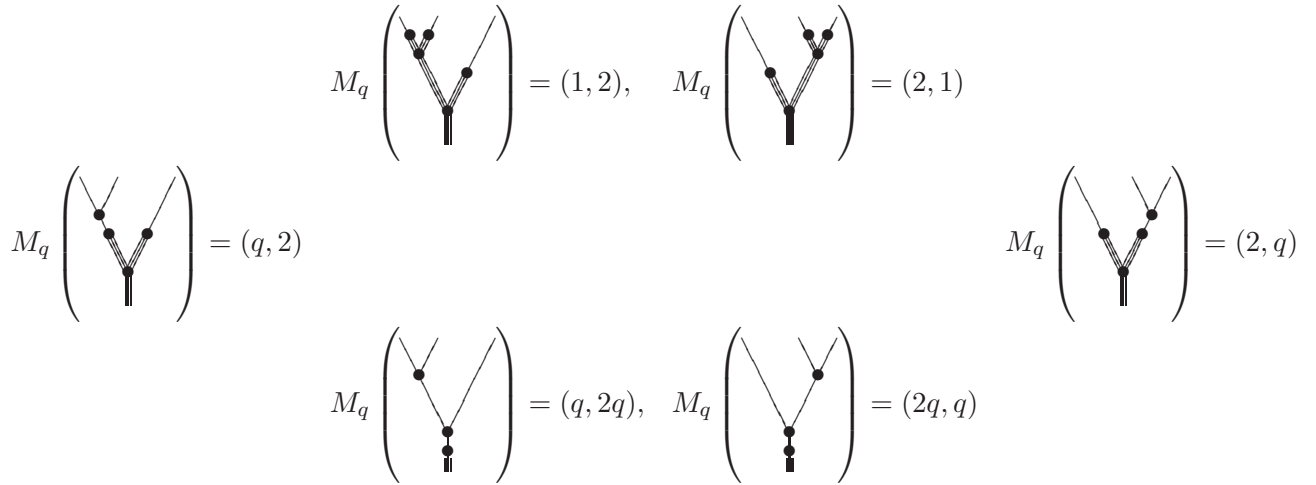
\begin{small}
\begin{tabular}{llllll}
&$ M_q\left(  \xy  0;/r.15pc/:
  (-10,14)*{\txt{\\\\}}="aa";
  (-2,14)*{}="ab";
  (2,14)*{}="ac";
  (10,14)*{}="ad";
  (-6,6)*=0{\bullet}="av1";
  (6,6)*=0{}="av2";
  (0,-6)*=0{\bullet}="av3";
  (0,-13)*=0{}="av4";
  (4,10)*=0{}="ava";
  (8,10)*=0{}="avb";
  (-4,10)*=0{\bullet}="ave";
  (-8,10)*=0{\bullet}="avf";
  (-4,2)*=0{}="avc" ;
  (4,2)*=0{\bullet}="avd" \ar@{};
 "aa" ;"avc" \ar@{-};"aa" ;"avc" \ar@{-};
 "ab" ;"av1" \ar@{-}  ;
  "ad" ;"avb" \ar@{-} \ar@{};
 "ave";"av1" \ar@{=} ;"ave";"av1" \ar@{=} ;
 "avf";"avc" \ar@{=} ;
 "avc" ;"av3" \ar@{=} ;
 "avd" ;"av3" \ar@{=} ;
 "av3" ;"av4" \ar@{=} \ar@{};
"avb";"av3" \ar@{-} ;
 "avb";"av3" \ar@{-} ;
 "avc" ;"av3" \ar@{-} ;
 "av3" ;"av4" \ar@{-} ;
 \endxy
\right) = (1,2),$ &  $M_q\left( \xy  0;/r.15pc/:
  (-10,14)*{\txt{\\\\}}="ba";
  (-2,14)*=0{}="bb";
  (2,14)*=0{}="bc";
  (10,14)*=0{}="bd";
  (-6,6)*=0{}="bv1";
  (6,6)*=0{\bullet}="bv2";
  (0,-6)*=0{\bullet}="bv3";
  (0,-13)*{}="bv4";
  (4,10)*=0{\bullet}="bva";
  (8,10)*=0{\bullet}="bvb";
  (-4,10)*=0{}="bve";
  (-8,10)*=0{}="bvf";
  (-4,2)*=0{\bullet}="bvc" ;
  (4,2)*=0{}="bvd" \ar@{};
 "ba" ;"bvc" \ar@{-};"ba" ;"bvc" \ar@{-};
  "bc" ;"bva" \ar@{-} ;
 "bd" ;"bvb" \ar@{-} \ar@{};
 "bva";"bv2" \ar@{=} ;"bva";"bv2" \ar@{=} ;
 "bvb";"bv3" \ar@{=} ;
 "bvc" ;"bv3" \ar@{=} ;
 "bv3" ;"bv4" \ar@{=} \ar@{};
 "bva";"bv2" \ar@{-} ;"bva";"bv2" \ar@{-} ;
 "bvb";"bv3" \ar@{-} ;
 "bvc" ;"bv3" \ar@{-} ;
 "bv3" ;"bv4" \ar@{-} ;
 \endxy\right) = (2,1)$ \\
   $M_q\left( \xy  0;/r.15pc/:
  (-10,14)*{\txt{\\\\}}="ca";
  (-2,14)*=0{}="cb";
  (2,14)*=0{}="cc";
  (10,14)*=0{}="cd";
  (-6,6)*=0{\bullet}="cv1";
  (6,6)*=0{}="cv2";
  (0,-6)*=0{\bullet}="cv3";
  (0,-13)*{}="cv4";
  (4,10)*=0{}="cva";
  (8,10)*=0{}="cvb";
  (-4,10)*=0{}="cve";
  (-8,10)*=0{}="cvf";
  (-4,2)*=0{\bullet}="cvc" ;
  (4,2)*=0{\bullet}="cvd" \ar@{};
 "ca" ;"cvc" \ar@{-};"ca" ;"cvc" \ar@{-};
 "cb" ;"cv1" \ar@{-}  ;
  "cd" ;"cvb" \ar@{-} \ar@{};
 "cvc" ;"cv3" \ar@{=} ;
 "cvc" ;"cv3" \ar@{=} ;
 "cvd" ;"cv3" \ar@{=} ;
 "cv3" ;"cv4" \ar@{=} \ar@{};
"cvb";"cv3" \ar@{-} ;
 "cvb";"cv3" \ar@{-} ;
 "cvc" ;"cv3" \ar@{-} ;
 "cv3" ;"cv4" \ar@{-} ;
 \endxy\right) = (q,2)$ &&&   $ M_q\left( \xy  0;/r.15pc/:
  (-10,14)*{\txt{\\\\}}="da";
  (-2,14)*=0{}="db";
  (2,14)*=0{}="dc";
  (10,14)*=0{}="dd";
  (-6,6)*=0{}="dv1";
  (6,6)*=0{\bullet}="dv2";
  (0,-6)*=0{\bullet}="dv3";
  (0,-13)*{}="dv4";
  (4,10)*=0{}="dva";
  (8,10)*=0{}="dvb";
  (-4,10)*=0{}="dve";
  (-8,10)*=0{}="dvf";
  (-4,2)*=0{\bullet}="dvc" ;
  (4,2)*=0{\bullet}="dvd" \ar@{};
 "da" ;"dvc" \ar@{-};"da" ;"dvc" \ar@{-};
  "dc" ;"dva" \ar@{-} ;
 "dd" ;"dvb" \ar@{-} \ar@{};
"dvd" ;"dv3" \ar@{=} ;"dvd" ;"dv3" \ar@{=} ;
 "dvc" ;"dv3" \ar@{=} ;
 "dv3" ;"dv4" \ar@{=} \ar@{};
 "dva";"dv2" \ar@{-} ;"dva";"dv2" \ar@{-} ;
 "dvb";"dv3" \ar@{-} ;
 "dvc" ;"dv3" \ar@{-} ;
 "dv3" ;"dv4" \ar@{-} ;
 \endxy\right) = (2,q)$ \\
 &   $ M_q\left(\xy  0;/r.15pc/:
  (-10,14)*{\txt{\\\\}}="ea";
  (-2,14)*=0{}="eb";
  (2,14)*=0{}="ec";
  (10,14)*=0{}="ed";
  (-6,6)*=0{\bullet}="ev1";
  (6,6)*=0{}="ev2";
  (0,-6)*=0{\bullet}="ev3";
  (0,-14)*{}="ev4";
  (0,-10)*{\bullet}="ev5";
  (4,10)*=0{}="eva";
  (8,10)*=0{}="evb";
  (-4,10)*=0{}="eve";
  (-8,10)*=0{}="evf";
  (-4,2)*=0{}="evc" ;
  (4,2)*=0{}="evd" \ar@{};
 "ea" ;"evc" \ar@{-};"ea" ;"evc" \ar@{-};
 "eb" ;"ev1" \ar@{-}  ;
  "ed" ;"evb" \ar@{-} \ar@{};
   "ev5" ;"ev4" \ar@{=}; "ev5" ;"ev4" \ar@{=} \ar@{};
"evb";"ev3" \ar@{-} ;
 "evb";"ev3" \ar@{-} ;
 "evc" ;"ev3" \ar@{-} ;
 "ev3" ;"ev4" \ar@{-} ;
 \endxy \right) = (q,2q),$ &  $      M_q\left( \xy  0;/r.15pc/:
  (-10,14)*{\txt{\\\\}}="fa";
  (-2,14)*=0{}="fb";
  (2,14)*=0{}="fc";
  (10,14)*=0{}="fd";
  (-6,6)*=0{}="fv1";
  (6,6)*=0{\bullet}="fv2";
  (0,-6)*=0{\bullet}="fv3";
  (0,-14)*{}="fv4";
  (0,-10)*{\bullet}="fv5";
  (4,10)*=0{}="fva";
  (8,10)*=0{}="fvb";
  (-4,10)*=0{}="fve";
  (-8,10)*=0{}="fvf";
  (-4,2)*=0{}="fvc" ;
  (4,2)*=0{}="fvd" \ar@{};
 "fa" ;"fvc" \ar@{-};"fa" ;"fvc" \ar@{-};
  "fc" ;"fva" \ar@{-} ;
 "fd" ;"fvb" \ar@{-} \ar@{};
  "fv5" ;"fv4" \ar@{=}; "fv5" ;"fv4" \ar@{=} \ar@{};
 "fva";"fv2" \ar@{-} ;"fva";"fv2" \ar@{-} ;
 "fvb";"fv3" \ar@{-} ;
 "fvc" ;"fv3" \ar@{-} ;
 "fv3" ;"fv4" \ar@{-} ;
 \endxy\right) = (2q,q)$ \\
\end{tabular}
\end{small}

$$$$
Thus for $q=\frac{1}{2}$ we have the six points
$\{(1,2),(2,1),(\frac{1}{2},2),(2,\frac{1}{2}),(\frac{1}{2},1),(1,\frac{1}{2})\}.$ Their convex
hull appears as follows:
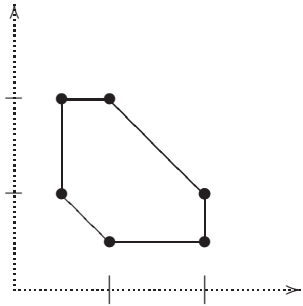
\begin{small}
$$
\xy 0;/r3pc/: (0,0);(0,3) \ar@{.>} ; (0,0);(0,3) \ar@{.>} ; (0,0); (3,0) \ar@{-} ;
  (1,2);(2,1) \ar@{-};(1,2)*=0{\bullet};(2,1)*=0{\bullet} \ar@{-};
  (1,2);(.5,2)*=0{\bullet}\ar@{-};
  (.5,2);(.5,1)*=0{\bullet} \ar@{-};
  (.5,1);(1,.5)*=0{\bullet} \ar@{-};
  (1,.5);(2,.5)*=0{\bullet} \ar@{-};
  (2,.5);(2,1)*=0{\bullet} \ar@{-}; \ar@{};
  (0,1) *=0{\txt{\textbf{--}}};
  (0,2) *=0{\txt{\textbf{--}}};
  (1,0) *=0{|};
  (2,0) *=0{|};
  \endxy
$$
\end{small}
\begin{example}
\end{example}
The list of vertices for ${\cal J}(4)$ based on painted binary trees with 4 leaves, for $q= {1
\over 2},$ is:
\newline

\begin{tabular}{lllll}
 (1, 2 ,3)&
 (1/2 ,2 ,3)&
 (1/2 ,2/2 ,3)&
 (1/2, 2/2 ,3/2)\\
 (2, 1, 3)&
 (2 ,1/2 ,3)&
 (2/2 ,1/2 ,3)&
 (2/2, 1/2 ,3/2)\\
 (3 ,1 ,2)&
 (3, 1/2, 2)&
 (3 ,1/2 ,2/2)&
 (3/2, 1/2 ,2/2)\\
 (3, 2, 1)&
 (3 ,2, 1/2)&
 (3 ,2/2, 1/2)&
 (3/2 ,2/2 ,1/2)\\
 (1 ,4 ,1)&
 (1/2, 4, 1)&
 (1, 4, 1/2)&
 (1/2, 4, 1/2)&
 (1/2, 4/2 ,1/2)\\
\end{tabular}
\newline
These are suggestively listed as a table where the first column is made up of the coordinates
calculated by Loday for $\cal{K}(4)$, which here correspond to trees with every trivalent node
entirely painted. The rows may be found by applying the factor $q$ to each coordinate in turn, in
order of increasing size of those coordinates. Here is the convex hull of these points, where we
see that each row of the table corresponds to shortest paths from the big pentagon to the small
one. Of course sometimes there are multiple such paths.
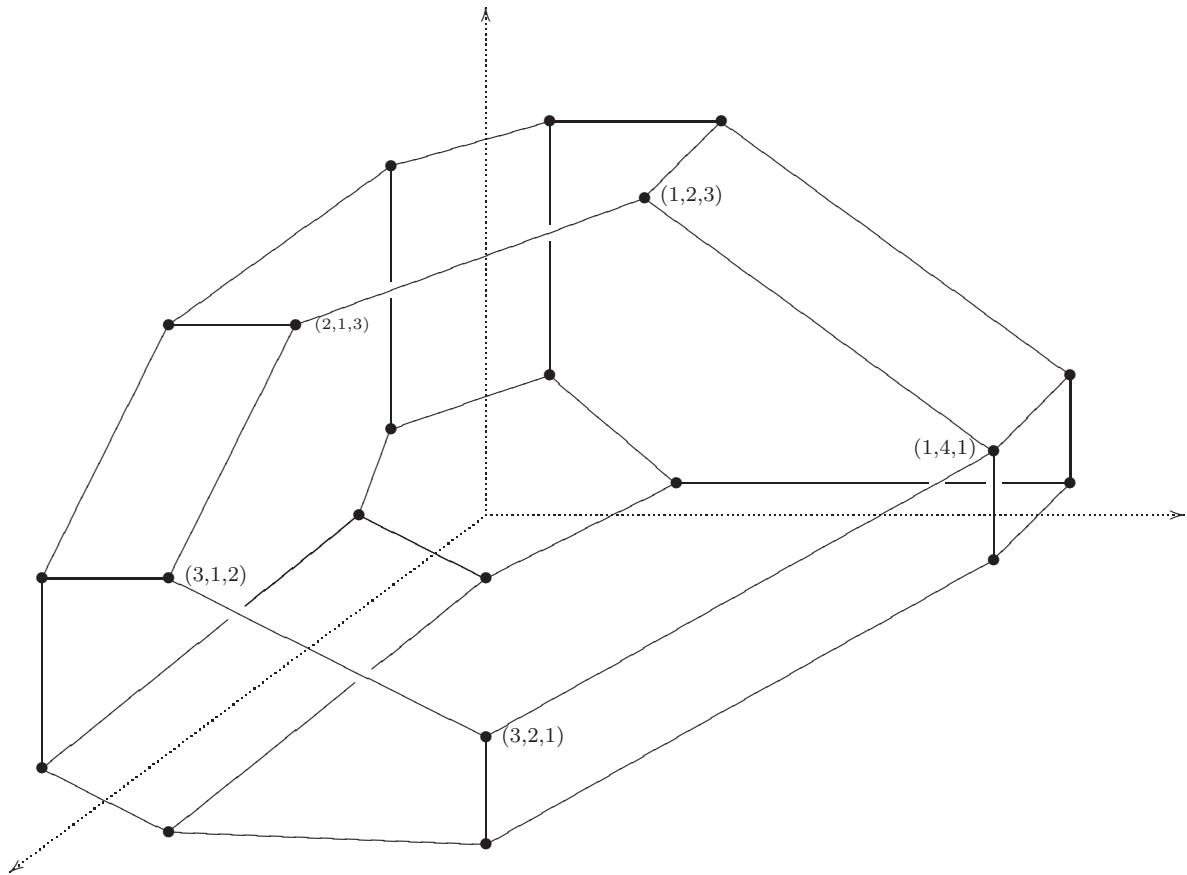
\newline

\begin{small}
$$
\xy 0;/r1pc/:
   (0,0);(0,16) \ar@{.>} ; (0,0);(0,16) \ar@{.>} ; (0,0);(-15,-11.25)\ar@{.>}; (0,0); (22,0) \ar@{-} ;
      (-4,0)*=0{\bullet}="1a";
  (0,-2)*=0{\bullet}="1b";
  (6,1)*=0{\bullet}="1c";
  (2,4.418)*=0{\bullet}="1d";
  (-3,2.7)*=0{\bullet}="1e";
  (-10,-2)*=0{\bullet}="2a";
  (-10,-2)*=0{^{(3,1,2)}\hspace{-.5in}};
  (0,-7)*=0{\bullet}="2b";
  (0,-7)*=0{_{(3,2,1)}\hspace{-.5in}};
  (16,2)*=0{\bullet}="2c";
  (16,2)*=0{^{(1,4,1)}\hspace{.5in}};
  (5,10)*=0{\bullet}="2d";
  (5,10)*=0{^{(1,2,3)}\hspace{-.5in}};
  (-6,6)*=0{\bullet}="2e";
  (-6,6)*=0{_{_{(2,1,3)}}\hspace{-.5in}};
  (-14,-2)*=0{\bullet}="3a";
  (-14,-8)*=0{\bullet}="3b";
  (-10,-10)*=0{\bullet}="3c";
  (0,-10.4)*=0{\bullet}="3d";
  (16,-1.418)*=0{\bullet}="3e";
  (18.418,1)*=0{\bullet}="3f";
  (18.418,4.418)*=0{\bullet}="3g";
  (7.418,12.418)*=0{\bullet}="3h";
  (2,12.418)*=0{\bullet}="3i";
  (-3,11)*=0{\bullet}="3j";
  (-10,6)*=0{\bullet}="3k";
 (14.222,1)*{\hole}="x"; (14.222,1)*{\hole}="x";
(16,1)*{\hole}="y";
  (-3,7.091)*{\hole}="z";
 (2,8.909)*{\hole}="w";
 (-7.846,-3.077)*{\hole}="p";
  (-3.846,-5.077)*{\hole}="q";
  "1a";"p" \ar@{-}; "1a";"p"\ar@{-};
  "p";"3b"  \ar@{-};
    "1b";"q" \ar@{-};
  "q";"3c" \ar@{-};
 "1c";"x" \ar@{-};
 "x";"y" \ar@{-};
 "y";"3f" \ar@{-};
 "1d";"w" \ar@{-};
 "w";"3i" \ar@{-};
 "1e";"z" \ar@{-};
 "z";"3j" \ar@{-};
 "2a";"3a" \ar@{-};
 "2b";"3d" \ar@{-};
  "2c";"3e" \ar@{-};
  "2c";"3g" \ar@{-};
 "2d";"3h" \ar@{-};
 "2e";"3k" \ar@{-};
"1a";"1b" \ar@{-}; "1a";"1b" \ar@{-};
 "1b";"1c" \ar@{-};
 "1c";"1d" \ar@{-};
 "1d";"1e" \ar@{-};
 "1e";"1a" \ar@{-};
 "2a";"2b" \ar@{-};
 "2b";"2c" \ar@{-};
 "2c";"2d" \ar@{-};
 "2d";"2e" \ar@{-};
 "2e";"2a" \ar@{-};
 "3a";"3b" \ar@{-};
  "3b";"3c" \ar@{-};
  "3c";"3d" \ar@{-};
  "3d";"3e" \ar@{-};
  "3e";"3f" \ar@{-};
  "3f";"3g" \ar@{-};
  "3g";"3h" \ar@{-};
  "3h";"3i" \ar@{-};
  "3i";"3j" \ar@{-};
  "3j";"3k" \ar@{-};
  "3k";"3a" \ar@{-};
     \endxy
$$
\end{small}
\newline

The largest pentagonal facet of this picture corresponds to the bottom pentagonal facet in the
drawing of ${\cal J}(4)$ on page 53 of \cite{sta2}, and to the pentagonal facet labeled $d_{(0,1)}$
in the diagram of ${\cal J}(4)$ in section 5 of \cite{umble}. Just turn the page 90 degrees
clockwise to see the picture of $\cal{J}(4)$ that is in the introduction of this paper.

To see a rotatable version of the convex hull which is the fourth multiplihedron, enter the
following homogeneous coordinates into the Web Demo of polymake (with option visual), at
\url{http://www.math.tu-berlin.de/polymake/index.html#apps/polytope}. Indeed polymake was
instrumental in the experimental phase of this research \cite{poly}.

\begin{align*}
POINTS\\
 1 ~1~ 2~ 3\\
 1~ 1/2~ 2 ~3\\
 1 ~1/2~ 2/2~ 3\\
 1 ~1/2 ~2/2~ 3/2\\
 1~ 2~ 1~ 3\\
 1 ~2~ 1/2 ~3\\
 1~ 2/2~ 1/2~ 3\\
 1~ 2/2 ~1/2~ 3/2\\
 1~ 3~ 1~ 2\\
 1~ 3~ 1/2~ 2\\
 1~ 3~ 1/2 ~2/2\\
 1~ 3/2~ 1/2~ 2/2\\
 1~ 3~ 2~ 1\\
 1~ 3~ 2~ 1/2\\
 1~ 3~ 2/2~ 1/2\\
 1~ 3/2~ 2/2~ 1/2\\
 1~ 1~ 4~ 1\\
 1~ 1/2~ 4 ~1\\
 1~ 1~ 4~ 1/2\\
 1~ 1/2 ~4 ~1/2\\
 1~ 1/2 ~4/2~ 1/2\\
\end{align*}
\newpage

\section{Spaces of painted trees}

Boardman and Vogt develop several versions of the space of colored or painted trees with $n$ leaves
with different uses for proving specific theorems about $A_{\infty}$ maps. We choose to focus on
one version which has the advantage of reflecting the intuitive dimension of the multiplihedra. The
points of this space are based on the binary painted trees with the three types of nodes pictured
in the introduction. The leaves are always colored by the domain $X$ (here we say unpainted),  and
the root is always colored by the range, $Y$ (here we say painted).

To get a point of the space each interior edge of a given binary painted tree with $n$ leaves is
assigned a value in $[0,1].$ The result is called a \emph{painted metric tree}. When none of the
trivalent nodes are painted (that is, disallowing the second node type), and with the equivalence
relations we will review shortly, this will become the space $SM{\cal U}(n,1)$ as defined in
\cite{BV1}. Allowing all three types of nodes gives the space
 $$HW({\cal U} \otimes {\cal L }_1)(n^0,1^1).$$

 (In \cite{BV1} the superscripts denote the colors, so this denotes that there are $n$
 inputs colored ``0'' and one output colored ``1.'' This is potentially
 confusing since these numbers are also used for edge lengths, and so in this paper we will
 denote coloring with the shaded edges and reserve the values to denote edge lengths.)

 We want to consider the retract of this space to the \emph{level trees}, denoted in \cite{BV1}
 $$LW({\cal U} \otimes {\cal L }_1)(n^0,1^1).$$
 The definition in \cite{BV1} simply declares that a level tree is either a tree that has one or
 zero nodes, or a tree that decomposes into level trees. The authors then unpack the definition
  a bit to demonstrate
 that the effect of their recursive requirement is to ensure that the the space of 2-leaved level
 trees has dimension 1. They declare in general that their space of $n$-leaved level trees will
 have the expected form, that is, will be homeomorphic to a closed $(n-1)$-dimensional ball.
We give here a specific way to realize a space of trees satisfying the recursive requirement and
having the expected form. Again the requirement will ensure that a decomposition of level trees
will always be into level trees.

 We will denote our version of the space of level trees with $n$ leaves by $LW{\cal U}(n).$
  It is defined  in Definition~\ref{lw} as the space of painted metric trees, after introducing
 relations on the lengths of edges.
\begin{definition}
We first describe a space corresponding to each painted binary tree. We denote it $W(t).$
  Edge lengths can be chosen freely from $[0,1]$ subject to the
 following
 conditions. At each trivalent node of a tree $t$ there are two subtrees with their root that node. The left subtree is
  defined by the tree with its rooted edge the left-hand branch of that node and the right subtree is likewise supported by the righthand branch.
  The conditions are that for each node of type (2) we have an equation relating the painted interior edge lengths
 of the left subtree and the right subtree (interior with respect to the original $t$). Let $u_1 \dots u_k$ be the lengths of the painted interior
 edges of the left subtree  and let $v_1 \dots v_j$ be the painted lengths of the right subtree. Let $p_u$ be the number of leaves
 of the left subtree and let $p_v$ be the number of leaves of the right subtree. The
 equation to be obeyed is
  $$\frac{1}{p_u}\smash{\sum\limits_{i=1}^k u_i} = \frac{1}{p_v}\smash{\sum\limits_{i=1}^j
 v_i}.$$
\end{definition}
 For example consider the edge lengths $u,v,x,y,z \in [0,1]$ assigned to the following tree:
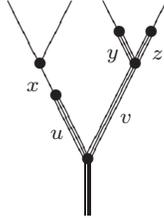
\begin{small}
 $$
\xy  0;/r.25pc/:
  (-10,20)*=0{}="a"; (-2,20)*=0{}="b";
  (2,20)*=0{}="c"; (10,20)*=0{}="d";
  (-6,12)*=0{\bullet}="v1"; (6,12)*=0{\bullet}="v2";
  (0,0)*=0{\bullet}="v3";
  (0,-7)*=0{}="v4";
  (4,16)*=0{\bullet}="va";
  (8,16)*=0{\bullet}="vb";
  (-4,8)*=0{\bullet}="vc" \ar@{};
 "a" ;"v1" \ar@{-};"a" ;"v1" \ar@{-}_{x};
 "v1" ;"vc" \ar@{-};
 "b" ;"v1" \ar@{-}  ;
 "c" ;"va" \ar@{-} ;
 "d" ;"vb" \ar@{-} \ar@{}_{y};
 "va";"v2" \ar@{=} ;"va";"v2" \ar@{=}^{z} ;
 "vb";"v2" \ar@{=}^v ;
 "v2" ;"v3" \ar@{=}_{u};
 "vc" ;"v3" \ar@{=} ;
 "v3" ;"v4" \ar@{=} \ar@{};
 %
 %
 "va";"v2" \ar@{-} ;"va";"v2" \ar@{-} ;
 "vb";"v3" \ar@{-} ;
 "vc" ;"v3" \ar@{-} ;
 "v3" ;"v4" \ar@{-} ;
 \endxy
$$
\end{small}

The relations on the lengths then are the equations:
 $$y=z  \text{ }\text{ }\text{ }\text{ }\text{ }\text{ and  }\text{ }\text{ }\text{ }\text{ }\text{ } \frac{1}{2}u=\frac{1}{2}(v+y+z).$$
Note that this will sometimes imply that  lengths of certain edges are forced to take values only
from $[0,p], p < 1.$ In \cite{BV1} the definition of the level trees is given by an inductive
property, which guarantees that decompositions of the trees will always be into level trees. This
seems equivalent to our requirement that the nodes be of types (1)-(6). The relations on edge
length serve to ensure that this requirement is preserved even as some edges go to zero.

Before describing how to glue together all these subspaces for different trees to create the entire
$LW{\cal U}(n)$ we show the following:
\begin{theorem}\label{dimtree}
The dimension of the subspace $W(t)$ of  $LW{\cal U}(n)$ corresponding to a given binary painted
tree is $n-1.$
\end{theorem}
\begin{proof}
After assigning variables to the internal edges and applying the relations, the total number of
free variables is at least the number of interior edges less the number of painted, type (2),
nodes. This difference is always one less than the number of leaves. To see that the constraining
equations really do reduce the number of free variables to $n-1,$ notice what the equations imply
about the painted interior edge lengths (the unpainted edge lengths are all free variables.)
Beginning at the painted nodes which are closest to the leaves and setting equal to zero one of the
two branches (a free variable) at each node it is seen that all the painted interior edge lengths
are forced to be zero. Thus each painted node can only contribute one free variable--the other
branch length must be dependent. Therefore, given a painted binary tree with $n$ leaves and $k$
internal edges, the space of points corresponding to the allowed choices for the edge values of
that tree is the intersection of an $(n-1)$-dimensional subspace of \textbf{R}$^k$ with $[0,1]^k.$
We see this simply by solving the system of homogeneous equations indicated by the type (2) nodes
and restricting our solution to the lengths in $[0,1].$

In fact, the intersection just described is an $(n-1)$-dimensional polytope in \textbf{R}$^k$. We
see that this is true since there is a point in the intersection for which each of the coordinates
is in the range $(0,{1 \over 2}]$. To see an example of such a point we consider edge lengths of
our binary tree such that the unpainted edges each have length ${1 \over 2}$ and such that the
painted edges have lengths in $(0,{1 \over 2}]$. To achieve the latter we
 begin at the first painted type (2) node above the root, and
consider the left and right subtrees. If the left subtree has only one painted edge we assign that
edge the length ${p \over 2n}$ where $p$ is the number of leaves of the left subtree; but if not
then we assign the root edge of the left  subtree the length ${p \over 4n}$. We do the same for the
right subtree, replacing $p$ with the number of leaves of the right subtree. This proceeds
inductively up the tree. At a given type (2) node if its left/right $p'$-leaved subtree has only
one painted edge we assign that edge the length ${p' \over d}$ where $d$ is the denominator of the
length assigned to the third edge (closest to the root) of the that node on the previous step; but
if not then we assign the root edge of the left/right subtree the length ${p' \over 2d}$ . This
produces a set of non-zero lengths which obey the relations and are all $\le {1 \over 2}.$ For
example:
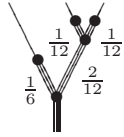
\begin{small}
$$
\xy  0;/r.15pc/:
  (-10,14)*{\txt{\\\\}}="ba";
  (-2,14)*=0{}="bb";
  (2,14)*=0{}="bc";
  (10,14)*=0{}="bd";
  (-6,6)*=0{}="bv1";
  (6,6)*=0{\bullet}="bv2";
  (0,-6)*=0{\bullet}="bv3";
  (0,-13)*{}="bv4";
  (4,10)*=0{\bullet}="bva";
  (8,10)*=0{\bullet}="bvb";
  (-4,10)*=0{}="bve";
  (-8,10)*=0{}="bvf";
  (-4,2)*=0{\bullet}="bvc" ;
  (4,2)*=0{}="bvd" \ar@{};
 "ba" ;"bvc" \ar@{-};"ba" ;"bvc" \ar@{-};
  "bc" ;"bva" \ar@{-} ;
 "bd" ;"bvb" \ar@{-} \ar@{}_{1 \over 12};
 "bva";"bv2" \ar@{=} ;"bva";"bv2" \ar@{=}^{1 \over 12} ;
 "bvb";"bv2" \ar@{=}^{2 \over 12} ;
 "bv2" ;"bv3" \ar@{=}_{1 \over 6} ;
 "bvc" ;"bv3" \ar@{=} ;
 "bv3" ;"bv4" \ar@{=} \ar@{};
 "bva";"bv2" \ar@{-} ;"bva";"bv2" \ar@{-} ;
 "bvb";"bv3" \ar@{-} ;
 "bvc" ;"bv3" \ar@{-} ;
 "bv3" ;"bv4" \ar@{-} ;
 \endxy
$$
\end{small}
\end{proof}

To describe the equivalence relations on our space we recall the trees with three additional
allowed node types. They correspond to the the node types (1), (2) and (3) in that they are painted
in similar fashion.
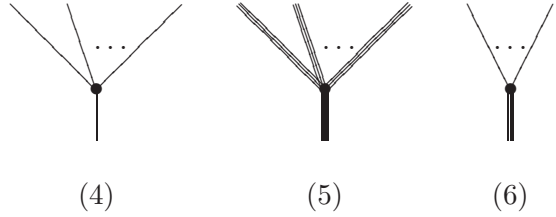
\begin{small}
 $$
\xy  0;/r.45pc/:
  (-14,20)*=0{}="a"; (-10,20)*=0{}="b";
  (-2,20)*=0{}="c"; (2,20)*=0{}="d";
  (6,20)*=0{}="e"; (14,20)*=0{}="f";
  (-7,17) *{\dots}; (9,17) *{\dots};
  (-8,14)*=0{\bullet}="v1"; (8,14)*=0{\bullet}="v2";
  (-8,8)*{\txt{\\\\(4)}}="v3";
  (8,8)*{\txt{\\\\(5)}}="v4"; \ar@{};
 (18,20)*=0{}="g"; (24,20)*=0{}="h";
 (21,17) *{\dots};
 (21,14)*=0{\bullet}="v5";
  (21,8)*{\txt{\\\\(6)}}="v6";
 "a" ;"v1" \ar@{-};"a" ;"v1" \ar@{-};
 "b" ;"v1" \ar@{-};
 "c" ;"v1" \ar@{-};
 "d" ;"v2" \ar@{-};
 "e" ;"v2" \ar@{-};
 "f" ;"v2" \ar@{-};
 "v1" ;"v3" \ar@{-};
 "v2" ;"v4" \ar@{-};
 "g" ;"v5" \ar@{-};
 "h" ;"v5" \ar@{-};
 "v5" ;"v6" \ar@{-};
 \ar@{} ;
"d" ;"v2" \ar@{=};"d" ;"v2" \ar@{=};
 "e" ;"v2" \ar@{=};
 "f" ;"v2" \ar@{=};
 "v2" ;"v4" \ar@{=};
 "v5" ;"v6" \ar@{=};
 \endxy
$$
\end{small}
These nodes each have subtrees supported by each of their branches in order from left to right. The
interior edges of each tree are again assigned lengths in $[0,1].$ The requirements on edge lengths
which we get from each node of type (5) of valence $j+1$ are the equalities:

$$\frac{1}{p_1}\smash{\sum\limits_{i=1}^{k_1} {u_1}_i} = \frac{1}{p_2}\smash{\sum\limits_{i=1}^{k_2}
 {u_2}_i} = \dots = \frac{1}{p_j}\smash{\sum\limits_{i=1}^{k_j} {u_j}_i}$$

where $k_1 \dots k_j$ are the numbers of painted internal edges of each of the $j$ subtrees, and
$p_1 \dots p_j$ are the numbers of leaves of each of the subtrees. Now we review the equivalence
relation on trees introduced in \cite{BV1}.
\begin{definition}\label{lw}
Now the space of painted metric trees with $n$ leaves $LW\cal{U}(n)$ is formed by first taking the
disjoint union of the $(n-1)$-dimensional polytopes $W(t),$ one polytope for
 each binary painted tree. Then it is
given the
 quotient topology (of the standard topology of the disjoint union of the polytopes in \textbf{R}$^k$)
  under the following equivalence relation:
 Two trees are equivalent if they reduce to the same
tree after shrinking to points their respective edges of length zero. This is why we call the
variable assigned to interior edges ``length'' in the first place. By ``same tree'' we mean
possessing the same painted tree structure and having the same lengths assigned to corresponding
edges. For example one pair of equivalence relations appears as follows:
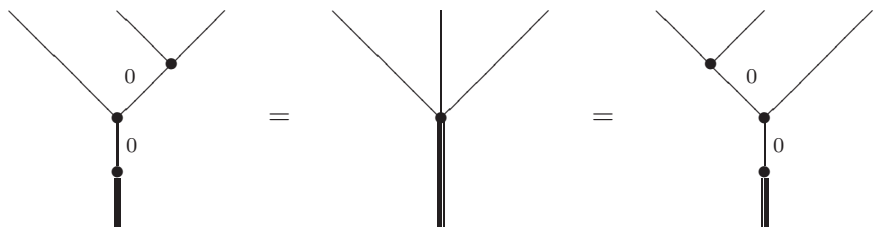
\begin{small}
$$
\xy  0;/r.85pc/:
  (-16,8)*=0{}="a"; (-12,8)*=0{}="b";
  (-8,8)*=0{}="c"; (-4,8)*=0{}="d";
  (0,8)*=0{}="e"; (4,8)*=0{}="f";
  (8,8)*=0{}="g"; (12,8)*=0{}="h";
  (16,8)*=0{}="i";
  (-6,4) *{=}; (6,4) *{=};
  (-10,6)*=0{\bullet}="v1"; (-12,4)*=0{\bullet}="v2";
  (-12,2)*{\bullet}="v3";
  (-12,0)*{}="v4";
  (10,6)*=0{\bullet}="v6"; (12,4)*=0{\bullet}="v7";
  (12,2)*{\bullet}="v8";
  (12,0)*{}="v9";
  (0,4)*=0{\bullet}="va"; (0,0)*=0{}="vb";
   \ar@{};
  "a" ;"v2" \ar@{-};"a" ;"v2" \ar@{-};
 "b" ;"v1" \ar@{-};
 "c" ;"v1" \ar@{-}_{0};
 "v1" ;"v2" \ar@{-}^{0};
 "v2" ;"v3" \ar@{-};
 "v3" ;"v4" \ar@{-};
 "d" ;"va" \ar@{-};
 "e" ;"va" \ar@{-};
 "f" ;"va" \ar@{-};
 "va" ;"vb" \ar@{-};
 "g" ;"v6" \ar@{-};
 "h" ;"v6" \ar@{-};
 "i" ;"v7" \ar@{-}^0;
 "v6" ;"v7" \ar@{-}^0;
 "v7" ;"v8" \ar@{-};
 "v8" ;"v9" \ar@{-};
 "v3" ;"v4" \ar@{=};"v3" ;"v4" \ar@{=};
 "va" ;"vb" \ar@{=};
 "v8" ;"v9" \ar@{=};
 \endxy
$$
\end{small}
\end{definition}
Note that an equivalence class of trees may always be represented by any one of several binary
trees, with only nodes of type (1), (2), and (3), since we can reduce the valence of nodes within
an equivalence class by introducing extra interior edges of length zero. However we often represent
the equivalence class with the unique tree that shows no zero edges. We refer to this as the
collapsed tree. Also note that the relations on the variable lengths of a tree which has some of
those lengths set to zero are precisely the relations on the variables of the collapsed tree
equivalent to it.

\begin{example}
$LW{\cal U}(1)$ is just a single point. Here is the space $LW{\cal U}(2),$ where we require $u=v:$
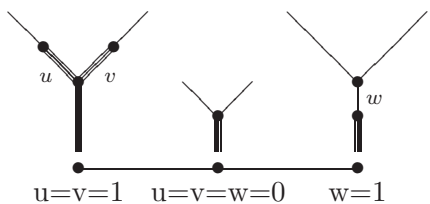
$$$$

\begin{small}
$$
\xy 0;/r.55pc/: (-12,8)*=0{}="a"; (-4,8)*=0{}="b";
  (4,8)*=0{}="e"; (12,8)*=0{}="f";
  (-2,4)*=0{}="c"; (2,4)*=0{}="d";
(-10,6)*=0{\bullet}="va";
 (-6,6)*=0{\bullet}="vb";
(-8,4)*=0{\bullet}="v1";
 (0,2)*=0{\bullet}="v2";
  (8,4)*=0{\bullet}="v3";
  (8,2)*=0{\bullet}="v4";
  (-8,0)*=0{}="r1";
(0,0)*=0{}="r2";
 (8,0)*=0{}="r3";
 (-8,-1)*=0{\bullet}="lp";
 (0,-1)*=0{\bullet}="cp";
 (8,-1)*=0{\bullet}="rp";
 (-8,-1)*=0{ \txt{\\\\u=v=1}};
 (0,-1)*=0{ \txt{\\\\u=v=w=0}};
 (8,-1)*=0{ \txt{\\\\w=1}};
"a" ;"v1" \ar@{-}; "a" ;"v1" \ar@{-};
 "b" ;"v1" \ar@{-};
 "v1" ;"r1" \ar@{-};
 "c" ;"v2" \ar@{-};
 "d" ;"v2" \ar@{-};
 "v2" ;"r2" \ar@{-};
 "e" ;"v3" \ar@{-};
 "f" ;"v3" \ar@{-}^w;
 "v3" ;"v4" \ar@{-};
 "v4" ;"r3" \ar@{-};
 "lp" ;"rp" \ar@{-};
 "va" ;"v1" \ar@{=}_u; "va" ;"v1" \ar@{=}^v;
 "vb" ;"v1" \ar@{=};
 "v1" ;"r1" \ar@{=};
 "v2" ;"r2" \ar@{=};
 "v4" ;"r3" \ar@{=};
 \endxy
$$
\end{small}

$$$$
$$$$

And here is the space $LW{\cal U}(3):$

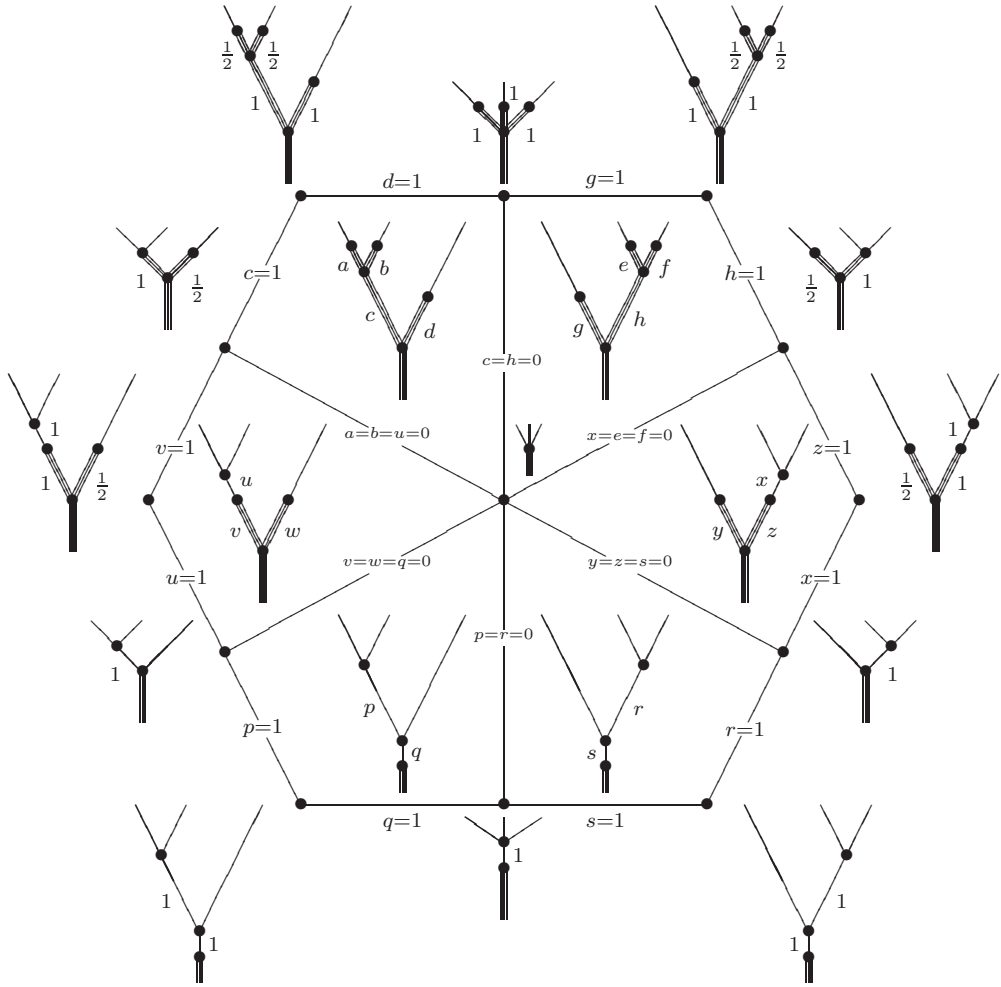
\begin{small}
$$
\xy 0;/r.2pc/:
 (-32,48);(0,48) *=0{\bullet}  \ar@{-}; (-32,48);(0,48) *=0{\bullet} ^{d=1} \ar@{-};
 (0,48);(32,48)*=0{\bullet} ^{g=1} \ar@{-};
  (32,48);(44,24)*=0{\bullet} |{h=1} \ar@{-};
  (44,24); (56,0)*=0{\bullet} |>>>>>>>>{z=1} \ar@{-};
   (56,0); (44,-24)*=0{\bullet} |{x=1} \ar@{-};
   (44,-24); (32,-48)*=0{\bullet} |{r=1} \ar@{-};
    (32,-48);(0,-48)*=0{\bullet} ^{s=1} \ar@{-};
    (0,-48); (-32,-48)*=0{\bullet} ^{q=1} \ar@{-};
     (-32,-48);(-44,-24)*=0{\bullet} |{p=1} \ar@{-};
     (-44,-24); (-56,0)*=0{\bullet} |{u=1} \ar@{-};
      (-56,0);(-44,24)*=0{\bullet} |<<<<<<<<{v=1} \ar@{-};
      (-44,24); (-32,48)*=0{\bullet} |{c=1} \ar@{};
      (0,0) *=0{\bullet} \ar@{-}|<<<<<<<<<<<<<<<<<{{}^{~~a=b=u=0}} ;
      (0,0);(-44,24)  \ar@{-}|<<<<<<<<<<<<<<<<<{{}_{~x=e=f=0}} ;
      (0,0);(44,24)  \ar@{-}|<<<<<<<<<<<<<<<<<{{}^{~v=w=q=0}};
      (0,0);(-44,-24)  \ar@{-}|<<<<<<<<<<<<<<<<<{{}^{~y=z=s=0}};
      (0,0);(44,-24)  \ar@{-}|<<<<<<<<<<<<<<<<<{{}^{\text{ }~c=h=0}};
      (0,0);(0,48)  \ar@{-}|<<<<<<<<<<<<<<<<<{{}^{p=r=0}};
      (0,0);(0,-48)  \ar@{=};
      (4,4);(4,8) \ar@{=}; (4,4);(4,8) *=0{\bullet}\ar@{-}; (4,4);(4,8) \ar@{-};
      (4,8);(2,12)  \ar@{-}; (4,8); (4,12) \ar@{-}; (4,8); (6,12) \ar@{-};
  (-26,44)*=0{}="aa";
  (-18,44)*=0{}="ab";
  (-14,44)*=0{}="ac";
  (-6,44)*=0{}="ad";
  (-22,36)*=0{\bullet}="av1";
  (-10,36)*=0{}="av2";
  (-16,24)*=0{\bullet}="av3";
  (-16,16)*=0{}="av4";
  (-12,40)*=0{}="ava";
  (-8,40)*=0{}="avb";
  (-20,40)*=0{\bullet}="ave";
  (-24,40)*=0{\bullet}="avf";
  (-20,32)*=0{}="avc";
  (-12,32)*=0{\bullet}="avd";
 "aa" ;"avc" \ar@{-};"aa" ;"avc" \ar@{-};
 "ab" ;"av1" \ar@{-}  ;
  "ad" ;"avb" \ar@{-} \ar@{};
 "ave";"av1" \ar@{=}^b;"ave";"av1" \ar@{=}_a ;
 "avf";"av1" \ar@{=}_c ;
 "av1" ;"av3" \ar@{=}^d ;
 "avd" ;"av3" \ar@{=} ;
 "av3" ;"av4" \ar@{=} \ar@{};
"avb";"av3" \ar@{-} ;
 "avb";"av3" \ar@{-} ;
 "avc" ;"av3" \ar@{-} ;
 "av3" ;"av4" \ar@{-} ;
  (6,44)*=0{}="ba";
  (14,44)*=0{}="bb";
  (18,44)*=0{}="bc";
  (26,44)*=0{}="bd";
  (10,36)*=0{}="bv1";
  (22,36)*=0{\bullet}="bv2";
  (16,24)*=0{\bullet}="bv3";
  (16,16)*=0{}="bv4";
  (20,40)*=0{\bullet}="bva";
  (24,40)*=0{\bullet}="bvb";
  (12,40)*=0{}="bve";
  (8,40)*=0{}="bvf";
  (12,32)*=0{\bullet}="bvc" ;
  (20,32)*=0{}="bvd" ;
 "ba" ;"bvc" \ar@{-};"ba" ;"bvc" \ar@{-};
  "bc" ;"bva" \ar@{-} ;
 "bd" ;"bvb" \ar@{-} \ar@{};
 "bva";"bv2" \ar@{=}_e ;"bva";"bv2" \ar@{=}^f ;
 "bvb";"bv2" \ar@{=}^h ;
 "bv2";"bv3" \ar@{=}_g ;
 "bvc" ;"bv3" \ar@{=} ;
 "bv3" ;"bv4" \ar@{=} \ar@{};
 "bva";"bv2" \ar@{-} ;"bva";"bv2" \ar@{-} ;
 "bvb";"bv3" \ar@{-} ;
 "bvc" ;"bv3" \ar@{-} ;
 "bv3" ;"bv4" \ar@{-} ;
  (-48,12)*=0{}="ca";
  (-40,12)*=0{}="cb";
  (-36,12)*=0{}="cc";
  (-28,12)*=0{}="cd";
  (-44,4)*=0{\bullet}="cv1";
  (-32,4)*=0{}="cv2";
  (-38,-8)*=0{\bullet}="cv3";
  (-38,-16)*=0{}="cv4";
  (-34,8)*=0{}="cva";
  (-30,8)*=0{}="cvb";
  (-42,8)*=0{}="cve";
  (-46,8)*=0{}="cvf";
  (-42,0)*=0{\bullet}="cvc" ;
  (-34,0)*=0{\bullet}="cvd" ;
 "ca" ;"cv1" \ar@{-};"ca" ;"cv1" \ar@{-}^u;
 "cv1" ;"cvc" \ar@{-};
 "cb" ;"cv1" \ar@{-}  ;
  "cd" ;"cvb" \ar@{-} \ar@{};
 "cvc" ;"cv3" \ar@{=}_v ;
 "cvc" ;"cv3" \ar@{=}^w ;
 "cvd" ;"cv3" \ar@{=} ;
 "cv3" ;"cv4" \ar@{=} \ar@{};
"cvb";"cv3" \ar@{-} ;
 "cvb";"cv3" \ar@{-} ;
 "cvc" ;"cv3" \ar@{-} ;
 "cv3" ;"cv4" \ar@{-} ;
  (28,12)*=0{}="da";
  (36,12)*=0{}="db";
  (40,12)*=0{}="dc";
  (48,12)*=0{}="dd";
  (32,4)*=0{}="dv1";
  (44,4)*=0{\bullet}="dv2";
  (38,-8)*=0{\bullet}="dv3";
  (38,-16)*=0{}="dv4";
  (42,8)*=0{}="dva";
  (46,8)*=0{}="dvb";
  (34,8)*=0{}="dve";
  (30,8)*=0{}="dvf";
  (34,0)*=0{\bullet}="dvc" ;
  (42,0)*=0{\bullet}="dvd" ;
 "da" ;"dvc" \ar@{-};"da" ;"dvc" \ar@{-};
  "dc" ;"dva" \ar@{-} ;
 "dd" ;"dvb" \ar@{-} \ar@{};
"dvd" ;"dv3" \ar@{=}^z ;"dvd" ;"dv3" \ar@{=}_y ;
 "dvc" ;"dv3" \ar@{=} ;
 "dv3" ;"dv4" \ar@{=} \ar@{};
 "dva";"dv2" \ar@{-} ;"dva";"dv2" \ar@{-} ;
 "dvb";"dv2" \ar@{-}_x ;
 "dv2";"dvd" \ar@{-} ;
  "dvd";"dv3" \ar@{-} ;
 "dvc" ;"dv3" \ar@{-} ;
 "dv3" ;"dv4" \ar@{-} ;
  (-26,-18)*=0{}="ea";
  (-18,-18)*=0{}="eb";
  (-14,-18)*=0{}="ec";
  (-6,-18)*=0{}="ed";
  (-22,-26)*=0{\bullet}="ev1";
  (-10,-26)*=0{}="ev2";
  (-16,-38)*=0{\bullet}="ev3";
  (-16,-46)*=0{}="ev4";
  (-16,-42)*=0{\bullet}="ev5";
  (-12,-22)*=0{}="eva";
  (-8,-22)*=0{}="evb";
  (-20,-22)*=0{}="eve";
  (-24,-22)*=0{}="evf";
  (-20,-30)*=0{}="evc" ;
  (-12,-30)*=0{}="evd" ;
 "ea" ;"evc" \ar@{-};"ea" ;"evc" \ar@{-};
 "eb" ;"ev1" \ar@{-}  ;
  "ed" ;"evb" \ar@{-} \ar@{};
   "ev5" ;"ev4" \ar@{=}; "ev5" ;"ev4" \ar@{=} \ar@{};
"evb";"ev3" \ar@{-} ;
 "evb";"ev3" \ar@{-} ;
 "evc" ;"ev1" \ar@{-}_p ;
 "ev1" ;"ev3" \ar@{-}^q ;
 "ev3" ;"ev5" \ar@{-} ;
 "ev5" ;"ev4" \ar@{-} ;
  (6,-18)*=0{}="fa";
  (14,-18)*=0{}="fb";
  (18,-18)*=0{}="fc";
  (26,-18)*=0{}="fd";
  (10,-26)*=0{}="fv1";
  (22,-26)*=0{\bullet}="fv2";
  (16,-38)*=0{\bullet}="fv3";
  (16,-46)*=0{}="fv4";
  (16,-42)*=0{\bullet}="fv5";
  (20,-22)*=0{}="fva";
  (24,-22)*=0{}="fvb";
  (12,-22)*=0{}="fve";
  (8,-22)*=0{}="fvf";
  (12,-30)*=0{}="fvc" ;
  (20,-30)*=0{}="fvd" ;
 "fa" ;"fvc" \ar@{-};"fa" ;"fvc" \ar@{-};
  "fc" ;"fva" \ar@{-} ;
 "fd" ;"fvb" \ar@{-} \ar@{};
  "fv5" ;"fv4" \ar@{=}; "fv5" ;"fv4" \ar@{=} \ar@{};
 "fva";"fv2" \ar@{-} ;"fva";"fv2" \ar@{-} ;
  "fvb";"fv2" \ar@{-}^r ;
 "fv2";"fv3" \ar@{-} ;
 "fvc" ;"fv1" \ar@{-} ;
 "fv1" ;"fv3" \ar@{-}_s ;
 "fv3" ;"fv5" \ar@{-} ;
 "fv5" ;"fv4" \ar@{-} ;
(   -44 ,   78  )      *=0{}="aa";
 (   -36 ,   78  )      *=0{}="ab";
  (   -32 ,   78  ) *=0{}="ac";
 (   -24 ,   78  )      *=0{}="ad";
  (   -40 ,   70  )      *=0{\bullet}="av1";
   (   -28 ,70  )   *=0{}="av2";
 (   -34 ,   58  )      *=0{\bullet}="av3";
  (   -34 ,   50  )      *=0{}="av4";
(   -30 ,   74  )      *=0{}="ava";
 (   -26 ,   74  )      *=0{}="avb";
  (   -38 ,   74  )*=0{\bullet}="ave";
 (   -42 ,   74  )      *=0{\bullet}="avf";
  (   -38 ,   66  )      *=0{}="avc";
(   -30 ,   66  )   *=0{\bullet}="avd";
 "aa" ;"avc" \ar@{-};"aa" ;"avc" \ar@{-};
 "ab" ;"av1" \ar@{-}  ;
  "ad" ;"avb" \ar@{-} \ar@{};
 "ave";"av1" \ar@{=}^{\frac{1}{2}};"ave";"av1" \ar@{=}_{\frac{1}{2}} ;
 "avf";"av1" \ar@{=}_1 ;
 "av1" ;"av3" \ar@{=}^1 ;
 "avd" ;"av3" \ar@{=} ;
 "av3" ;"av4" \ar@{=} \ar@{};
"avb";"av3" \ar@{-} ;
 "avb";"av3" \ar@{-} ;
 "avc" ;"av3" \ar@{-} ;
 "av3" ;"av4" \ar@{-} ;
 (   24  ,   78  )   *=0{}="ba";
(   32  ,   78  )   *=0{}="bb";
 (   36  ,   78  )   *=0{}="bc";
  (   44  ,   78  )   *=0{}="bd";
   (28  ,   70  )   *=0{}="bv1";
 (   40  ,   70  )   *=0{\bullet}="bv2";
  (   34  ,   58  )*=0{\bullet}="bv3";
 (   34  ,   50  )   *=0{}="bv4";
  (   38  ,   74  )   *=0{\bullet}="bva";
   (   42,   74  )   *=0{\bullet}="bvb";
 (   30  ,   74  )   *=0{}="bve";
  (   26  ,   74  )   *=0{}="bvf";
   (30  ,   66  )   *=0{\bullet}="bvc" ;
 (   38  ,   66  )   *=0{}="bvd" ;
 "ba" ;"bvc" \ar@{-};"ba" ;"bvc" \ar@{-};
  "bc" ;"bva" \ar@{-} ;
 "bd" ;"bvb" \ar@{-} \ar@{};
 "bva";"bv2" \ar@{=}_{\frac{1}{2}} ;"bva";"bv2" \ar@{=}^{\frac{1}{2}} ;
 "bvb";"bv2" \ar@{=}^1 ;
 "bv2";"bv3" \ar@{=}_1 ;
 "bvc" ;"bv3" \ar@{=} ;
 "bv3" ;"bv4" \ar@{=} \ar@{};
 "bva";"bv2" \ar@{-} ;"bva";"bv2" \ar@{-} ;
 "bvb";"bv3" \ar@{-} ;
 "bvc" ;"bv3" \ar@{-} ;
 "bv3" ;"bv4" \ar@{-} ;
(   -78 ,   20  )            *=0{}="ca";
 (   -70 ,   20  )            *=0{}="cb";
  (   -66 ,   20  )*=0{}="cc";
 (   -58 ,   20  )            *=0{}="cd";
  (   -74 ,   12  )*=0{\bullet}="cv1";
 (   -62 ,   12  )            *=0{}="cv2";
  (   -68 ,   0   )*=0{\bullet}="cv3";
 (   -68 ,   -8  )   *=0{}="cv4";
  (   -64 ,   16  )            *=0{}="cva";
   (-60 ,   16  )            *=0{}="cvb";
 (   -72 ,   16  )      *=0{}="cve";
  (   -76 ,   16  )*=0{}="cvf";
 (   -72 ,   8   )   *=0{\bullet}="cvc" ;
  (   -64 ,   8   )*=0{\bullet}="cvd" ;
 "ca" ;"cv1" \ar@{-};"ca" ;"cv1" \ar@{-}^1;
 "cv1" ;"cvc" \ar@{-};
 "cb" ;"cv1" \ar@{-}  ;
  "cd" ;"cvb" \ar@{-} \ar@{};
 "cvc" ;"cv3" \ar@{=}_1 ;
 "cvc" ;"cv3" \ar@{=}^{\frac{1}{2}} ;
 "cvd" ;"cv3" \ar@{=} ;
 "cv3" ;"cv4" \ar@{=} \ar@{};
"cvb";"cv3" \ar@{-} ;
 "cvb";"cv3" \ar@{-} ;
 "cvc" ;"cv3" \ar@{-} ;
 "cv3" ;"cv4" \ar@{-} ;
(   58  ,   20  )      *=0{}="da";
 (   66  ,   20  )      *=0{}="db";
  (   70  ,   20  )*=0{}="dc";
 (   78  ,   20  )      *=0{}="dd";
  (   62  ,   12  )      *=0{}="dv1";
   (   74  ,   12)   *=0{\bullet}="dv2";
 (   68  ,   0   )      *=0{\bullet}="dv3";
  (   68  ,   -8  )*=0{}="dv4";
 (   72  ,   16  )      *=0{}="dva";
  (   76  ,   16  )      *=0{}="dvb";
   (   64  ,   16)      *=0{}="dve";
 (   60  ,   16  )      *=0{}="dvf";
  (   64  ,   8   )      *=0{\bullet}="dvc" ;
(   72  ,   8   )   *=0{\bullet}="dvd" ;
 "da" ;"dvc" \ar@{-};"da" ;"dvc" \ar@{-};
  "dc" ;"dva" \ar@{-} ;
 "dd" ;"dvb" \ar@{-} \ar@{};
"dvd" ;"dv3" \ar@{=}^1 ;"dvd" ;"dv3" \ar@{=}_{\frac{1}{2}} ;
 "dvc" ;"dv3" \ar@{=} ;
 "dv3" ;"dv4" \ar@{=} \ar@{};
 "dva";"dv2" \ar@{-} ;"dva";"dv2" \ar@{-} ;
 "dvb";"dv2" \ar@{-}_1 ;
 "dv2";"dvd" \ar@{-} ;
  "dvd";"dv3" \ar@{-} ;
 "dvc" ;"dv3" \ar@{-} ;
 "dv3" ;"dv4" \ar@{-} ;
(   -58 ,   -48 )   *=0{}="ea";
 (   -50 ,   -48 )   *=0{}="eb";
  (   -46 ,   -48 )   *=0{}="ec";
   (-38 ,   -48 )   *=0{}="ed";
 (   -54 ,   -56 )   *=0{\bullet}="ev1";
  (   -42 ,   -56 )*=0{}="ev2";
 (   -48 ,   -68 )   *=0{\bullet}="ev3";
  (   -48 ,   -76 )   *=0{}="ev4";
   (   -48 ,-72 )   *=0{\bullet}="ev5";
 (   -44 ,   -52 )   *=0{}="eva";
  (   -40 ,   -52 )   *=0{}="evb";
   (-52 ,   -52 )   *=0{}="eve";
 (   -56 ,   -52 )   *=0{}="evf";
  (   -52 ,   -60 )   *=0{}="evc" ;
   (-44 ,   -60 )   *=0{}="evd" ;
 "ea" ;"evc" \ar@{-};"ea" ;"evc" \ar@{-};
 "eb" ;"ev1" \ar@{-}  ;
  "ed" ;"evb" \ar@{-} \ar@{};
   "ev5" ;"ev4" \ar@{=}; "ev5" ;"ev4" \ar@{=} \ar@{};
"evb";"ev3" \ar@{-} ;
 "evb";"ev3" \ar@{-} ;
 "evc" ;"ev1" \ar@{-}_1 ;
 "ev1" ;"ev3" \ar@{-}^1 ;
 "ev3" ;"ev5" \ar@{-} ;
 "ev5" ;"ev4" \ar@{-} ;
  (   38  ,   -48 )   *=0{}="fa";
(   46  ,   -48 )   *=0{}="fb";
 (   50  ,   -48 )   *=0{}="fc";
  (   58  ,   -48 )   *=0{}="fd";
   (42  ,   -56 )   *=0{}="fv1";
 (   54  ,   -56 )   *=0{\bullet}="fv2";
  (   48  ,   -68 )*=0{\bullet}="fv3";
 (   48  ,   -76 )   *=0{}="fv4";
  (   48  ,   -72 )   *=0{\bullet}="fv5";
   (   52,   -52 )   *=0{}="fva";
 (   56  ,   -52 )   *=0{}="fvb";
  (   44  ,   -52 )   *=0{}="fve";
   (   40,   -52 )   *=0{}="fvf";
 (   44  ,   -60 )   *=0{}="fvc" ;
  (   52  ,   -60 )   *=0{}="fvd" ;
 "fa" ;"fvc" \ar@{-};"fa" ;"fvc" \ar@{-};
  "fc" ;"fva" \ar@{-} ;
 "fd" ;"fvb" \ar@{-} \ar@{};
  "fv5" ;"fv4" \ar@{=}; "fv5" ;"fv4" \ar@{=} \ar@{};
 "fva";"fv2" \ar@{-} ;"fva";"fv2" \ar@{-} ;
  "fvb";"fv2" \ar@{-}^1 ;
 "fv2";"fv3" \ar@{-} ;
 "fvc" ;"fv1" \ar@{-} ;
 "fv1" ;"fv3" \ar@{-}_1 ;
 "fv3" ;"fv5" \ar@{-} ;
 "fv5" ;"fv4" \ar@{-} ;
 (-8,66)*=0{}="a"; (8,66)*=0{}="b";
 (0,66)*=0{}="c";
  (-4,62)*=0{\bullet}="va";
  (0,62)*=0{\bullet}="vc";
 (4,62)*=0{\bullet}="vb";
(0,58)*=0{\bullet}="v1";
  (0,50)*=0{}="r1";
"a" ;"v1" \ar@{-}; "a" ;"v1" \ar@{-};
  "c" ; "v1" \ar@{-};
 "b" ;"v1" \ar@{-};
 "v1" ;"r1" \ar@{-};
 "va" ;"v1" \ar@{=}_1; "va" ;"v1" \ar@{=}^1_>>>>>>>>>>>1;
 "vb" ;"v1" \ar@{=};
 "vc" ; "v1" \ar@{=};
 "v1" ;"r1" \ar@{-};
(45,43)*=0{}="a"; (61,43)*=0{}="b";
 (53,43)*=0{}="c";
  (49,39)*=0{\bullet}="va";
 (57,39)*=0{\bullet}="vb";
(53,35)*=0{\bullet}="v1";
  (53,27)*=0{}="r1";
"a" ;"v1" \ar@{-}; "a" ;"v1" \ar@{-};
  "c" ; "vb" \ar@{-};
 "b" ;"v1" \ar@{-};
 "v1" ;"r1" \ar@{-};
 "va" ;"v1" \ar@{=}_{\frac{1}{2}}; "va" ;"v1" \ar@{=}^1;
 "vb" ;"v1" \ar@{=};
 "v1" ;"r1" \ar@{-};
(-45,43)*=0{}="a";
 (-61,43)*=0{}="b";
 (-53,43)*=0{}="c";
  (-49,39)*=0{\bullet}="va";
 (-57,39)*=0{\bullet}="vb";
(-53,35)*=0{\bullet}="v1";
  (-53,27)*=0{}="r1";
"a" ;"v1" \ar@{-}; "a" ;"v1" \ar@{-};
  "c" ; "vb" \ar@{-};
 "b" ;"v1" \ar@{-};
 "v1" ;"r1" \ar@{-};
 "va" ;"v1" \ar@{=}^{\frac{1}{2}}; "va" ;"v1" \ar@{=}_1;
 "vb" ;"v1" \ar@{=};
 "v1" ;"r1" \ar@{-};
 (-6,-50)*=0{}="a"; (6,-50)*=0{}="b";
 (0,-50)*=0{}="c";
    (0,-54)*=0{\bullet}="vc";
 (0,-58)*=0{\bullet}="v1";
  (0,-66)*=0{}="r1";
"a" ;"vc" \ar@{-}; "a" ;"vc" \ar@{-};
  "c" ; "vc" \ar@{-};
 "b" ;"vc" \ar@{-};
 "v1" ;"r1" \ar@{-};
  \ar@{-}^1;
 "vc" ; "v1" \ar@{=};
 "v1" ;"r1" \ar@{-};
(49,-19)*=0{}="a"; (65,-19)*=0{}="b";
 (57,-19)*=0{}="c";
  (53,-23)*=0{}="va";
   (61,-23)*=0{\bullet}="vb";
(57,-27)*=0{\bullet}="v1";
  (57,-35)*=0{}="r1";
"a" ;"v1" \ar@{-}; "a" ;"v1" \ar@{-};
  "c" ; "vb" \ar@{-};
 "b" ;"v1" \ar@{-};
 "v1" ;"r1" \ar@{-};
 "va" ;"v1" \ar@{-}; "va" ;"v1" \ar@{-}^1;
 "vb" ;"v1" \ar@{=};
  "v1" ;"r1" \ar@{-};
(-49,-19)*=0{}="a"; (-65,-19)*=0{}="b";
 (-57,-19)*=0{}="c";
  (-53,-23)*=0{}="va";
   (-61,-23)*=0{\bullet}="vb";
(-57,-27)*=0{\bullet}="v1";
  (-57,-35)*=0{}="r1";
"a" ;"v1" \ar@{-}; "a" ;"v1" \ar@{-};
  "c" ; "vb" \ar@{-};
 "b" ;"v1" \ar@{-};
 "v1" ;"r1" \ar@{-};
 "va" ;"v1" \ar@{-}; "va" ;"v1" \ar@{-}_1;
 "vb" ;"v1" \ar@{=};
  "v1" ;"r1" \ar@{-};
 \endxy
$$
\end{small}
Note that the equations which the variables in $LW{\cal U}(3)$ must obey are:
\begin{align*}
a=b \text{ }\text{ }\text{ }\text{ }\text{ }\text{ and  }\text{ }\text{ }\text{ }\text{ }\text{
}d=\frac{1}{2}(a+b+c)\\
 e=f \text{ }\text{ }\text{ }\text{ }\text{ }\text{ and  }\text{ }\text{ }\text{ }\text{ }\text{ } g = \frac{1}{2}(e+f+h)
 \\
 w= \frac{1}{2}v \text{ }\text{ }\text{ }\text{ }\text{ }\text{ and  }\text{ }\text{ }\text{ }\text{ }\text{ } y=
 \frac{1}{2}z
\end{align*}
\end{example}
In \cite{Mau} the space of painted metric trees (bicolored metric ribbon trees) is described in a
slightly different way. First, the trees are not drawn with painted edges, but instead the nodes of
type (3) are indicated by color, and the edges between the root and those nodes can be assumed to
be painted. The correspondence is clear: for example,
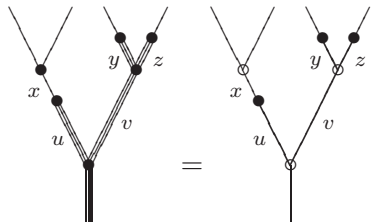
\begin{small}
 $$
\xy  0;/r.25pc/:
  (-10,20)*=0{}="a"; (-2,20)*=0{}="b";
  (2,20)*=0{}="c"; (10,20)*=0{}="d";
  (-6,12)*=0{\bullet}="v1"; (6,12)*=0{\bullet}="v2";
  (0,0)*=0{\bullet}="v3";
  (0,-7)*=0{}="v4";
  (4,16)*=0{\bullet}="va";
  (8,16)*=0{\bullet}="vb";
  (-4,8)*=0{\bullet}="vc" \ar@{};
 "a" ;"v1" \ar@{-};"a" ;"v1" \ar@{-}_{x};
 "v1" ;"vc" \ar@{-};
 "b" ;"v1" \ar@{-}  ;
 "c" ;"va" \ar@{-} ;
 "d" ;"vb" \ar@{-} \ar@{}_{y};
 "va";"v2" \ar@{=} ;"va";"v2" \ar@{=}^{z} ;
 "vb";"v2" \ar@{=}^v ;
 "v2" ;"v3" \ar@{=}_{u};
 "vc" ;"v3" \ar@{=} ;
 "v3" ;"v4" \ar@{=} \ar@{};
 %
 %
 "va";"v2" \ar@{-} ;"va";"v2" \ar@{-} ;
 "vb";"v3" \ar@{-} ;
 "vc" ;"v3" \ar@{-} ;
 "v3" ;"v4" \ar@{-} ;
 \endxy
 =
\xy  0;/r.25pc/:
  (-10,20)*=0{}="a"; (-2,20)*=0{}="b";
  (2,20)*=0{}="c"; (10,20)*=0{}="d";
  (-6,12)*=0{\circ}="v1"; (6,12)*=0{\circ}="v2";
  (0,0)*=0{\circ}="v3";
  (0,-7)*=0{}="v4";
  (4,16)*=0{\bullet}="va";
  (8,16)*=0{\bullet}="vb";
  (-4,8)*=0{\bullet}="vc" \ar@{};
 "a" ;"v1" \ar@{-};"a" ;"v1" \ar@{-}_{x};
 "v1" ;"vc" \ar@{-};
 "b" ;"v1" \ar@{-}  ;
 "c" ;"va" \ar@{-} ;
 "d" ;"vb" \ar@{-} \ar@{}_{y};
 "va";"v2" \ar@{-} ;"va";"v2" \ar@{-}^{z} ;
 "vb";"v2" \ar@{-}^v ;
 "v2" ;"v3" \ar@{-}_{u};
 "vc" ;"v3" \ar@{-} ;
 "v3" ;"v4" \ar@{-} \ar@{};
 %
 %
 "va";"v2" \ar@{-} ;"va";"v2" \ar@{-} ;
 "vb";"v3" \ar@{-} ;
 "vc" ;"v3" \ar@{-} ;
 "v3" ;"v4" \ar@{-} ;
 \endxy
$$
\end{small}
Secondly, the relations required of the painted lengths are different. In \cite{Mau} it is required
that the sum of the painted lengths along a path from the root to a leaf must always be the same.
For example, for the above tree, the new relations obeyed in \cite {Mau} are $u= v+y = v+z.$ This
provides the same dimension of $n-1$ for the space associated to a single binary tree with $n$
leaves as found in Theorem~\ref{dimtree} in this paper.

Thirdly the topology on the space of painted metric trees with $n$ leaves is described by first
assigning lengths in $(0,\infty)$ and then defining the limit as some lengths in a given tree
approach 0 as being the tree with those edges collapsed. This topology clearly is equivalent to the
definition as a quotient space given here and in \cite{BV1}. Thus we can use the results of
\cite{Mau} to show the following:
\begin{lemma}\label{ball}
The space $LW{\cal U}(n)$ is homeomorphic  to the closed ball in \textbf{R}$^{n-1}.$
\end{lemma}
\begin{proof} (1)
In \cite{Mau} it is shown that the entire space of painted trees with $n$ leaves with lengths in
$[0,\infty)$ is homeomorphic to $\textbf{R}^{n-1}_+ \cup \textbf{0}$. (This is done via a
homeomorphism to the space of quilted disks.) Thus if the lengths are restricted to lie in $[0,1]$
then the resulting space is homeomorphic to the closed ball in \textbf{R}$^{n-1}.$
\end{proof}
However, we think it valuable to see how the homeomorphism from the entire space of trees to the
convex polytope might actually be constructed piecewise from smaller homeomorphisms based on
specific $n$-leaved trees.
\begin{proof} (2)
We will use the Alexander trick, which is the theorem that states that any homeomorphism from the
bounding sphere of one disk to another bounding sphere of a second disk may be extended to the
entire disks. We are using this to construct a homeomorphism $\varphi$ from the convex hull
realization of $\cal{J}(n)$ to $LWU(n).$  First we consider the barycentric subdivision of the
former $(n-1)$-dimensional polytope. Recalling that each face of $\cal{J}(n)$ is associated with a
specific painted $n$-leaved tree $t$, we associate that same tree to the respective barycenter
denoted $v(t)$.

We will be creating $\varphi$ inductively. We begin by defining it on the crucial barycenters. The
barycenter of the entire polytope $\cal{J}(n)$ is associated to the painted corolla, and should be
mapped to the equivalence class represented by the corolla--that is, the class of trees with all
zero length interior edges.

The barycenters of facets of $\cal{J}(n)$ are each associated to a lower or upper tree. Since the
relations on variable edge lengths are preserved by collapsing zero edges, we can see that each of
these facet trees correspond to a one dimensional subset of the space of metric trees. Upper trees
have one fewer relation than the number of painted interior edges (and no other interior edges)
while lower trees have a single interior edge. The barycenters of lower facets are mapped to the
class represented by their respective tree with edge length 1. The barycenters of upper facets are
mapped to the class represented by their respective trees with maximal edge lengths. The maximal
lengths are found by choosing an edge with maximal valence type (6) node, and assigning length 1 to
that edge. The other lengths are then determined. Examples of this are shown by the facets of the
hexagon that is $LWU(3)$ above.

Now consider a particular binary painted tree $t$, associated to a vertex $v(t)=M_q(t)$ of
$\cal{J}(n).$ The simplicial complex made up of all the simplices in the barycentric subdivision
which contain $v(t)$ we denote $U(t).$ $U(t)$ is spanned by the vertices $v(t')$ for all $t'<t.$
Recall that $t' < t$ denotes that $t'$ refines $t,$ which means that $t$ results from the collapse
of some of the internal edges of $t'$. $U(t)$ is homeomorphic to the $(n-1)$-disk. Next we will
extend our choice of images of the facet barycenters for facets adjacent to $v(t)$ to a
homeomorphism $\varphi_t : U(t) \to W(t).$ This extension will be done incrementally where the
increments correspond to the refinement of trees, so that the piecewise defined mapping $\varphi(x)
= \varphi_t(x)~ ;~ x\in U(t)$ (with one piece defined on $U(t)$ for each binary $n$-leaved $t$)
will be well defined, 1-1, and onto $LWU(n).$ $U(t)$ for a particular 4-leaved tree is pictured as
a subset of the convex hull realization of $\cal{J}(4)$ just following this proof.

The incremental construction of our homeomorphism $\varphi_t$ is by way of subdividing the
respective boundaries of $U(t)$ and $W(t)$ based upon tree refinement.  For each tree $t' < t$, let
$p$ be the number of free variables in the metric version of $t'$ (so $n-(p+1)$ is the dimension of
the face associated to $t'$), and define $U(t')$ to be the sub-complex of $p$-simplices of $U(t)$
spanned by $v(t')$ and all the $v(t'')$ for $t'' < t'.$ $U(t')$ is a $p$-disk by construction. Also
define $W(t')$ to be the sub-space of the boundary of $W(t)$ given by all those equivalence classes
which can be represented by a metric version of $t'$, with interior edge lengths in [0,1]. By a
parallel argument to Theorem~\ref{dimtree} $W(t')$ is also a  $p$-disk.

To establish the base case we consider a facet barycenter (with associated tree $t' < t$). The
barycenter $v(t')$ and the barycenter of $\cal{J}(n)$ form a copy of $S^0$ bounding the 1-simplex
$U(t')$. Now the 1-dimensional subset $W(t')$ of the boundary of $W(t)$  is made up of equivalence
classes of trees represented by metric versions of $t'.$ The boundary of this 1-disk is the copy of
$S^0$ given by the tree with all zero lengths and the tree with maximal length. Thus we can extend
that choice of images made above to a homeomorphism $\varphi_{t'}$ of the 1-disks for each $t'.$

For an arbitrary tree $t'$ the boundary of $U(t')$ is a $(p-1)$-spherical simplicial complex that
is made up of two $(p-1)$-disks. The first \emph{interior} disk is the union of $U(t'')$ for $t'' <
t'.$ Each $(p-1)$-simplex in this first disk contains the barycenter of $\cal{J}(n).$ Each
$(p-1)$-simplex in the second \emph{exterior} disk contains $v(t).$ The shared boundary of the two
disks is a $(p-2)$-sphere. The boundary of $W(t')$ is also made up of two $(p-1)$-disks. The first
disk is the union of $W(t'')$ for $t'' < t'.$ The second disk is the collection of equivalence
classes of metric trees represented by $t'$ with at least one edge set equal to 1. Now we can build
$\varphi_t$ inductively by assuming it to be defined on the disks: $U(t'') \to W(t'')$ for all
trees $t''< t'.$ This assumed mapping may then be restricted to a homeomorphism of the
$(p-2)$-spheres that are the respective boundaries of the interior disks, which in turn can then be
extended to the exterior disks and thus the entire  $(p-1)$-spherical boundaries of $U(t')$ and
$W(t').$ From there the homeomorphism can be extended to the entire $p$-disks: $U(t') \to W(t').$
This continues inductively until, after the last extension, the resulting homeomorphism is called
$\varphi_t: U(t)\to W(t).$

Now by construction the map $\varphi : \cal{J}(n)\to LWU(n)$ given by $\varphi(x) = \varphi_t(x)
~;~ x\in U(t)$ is well defined, continuous, bijective and open.



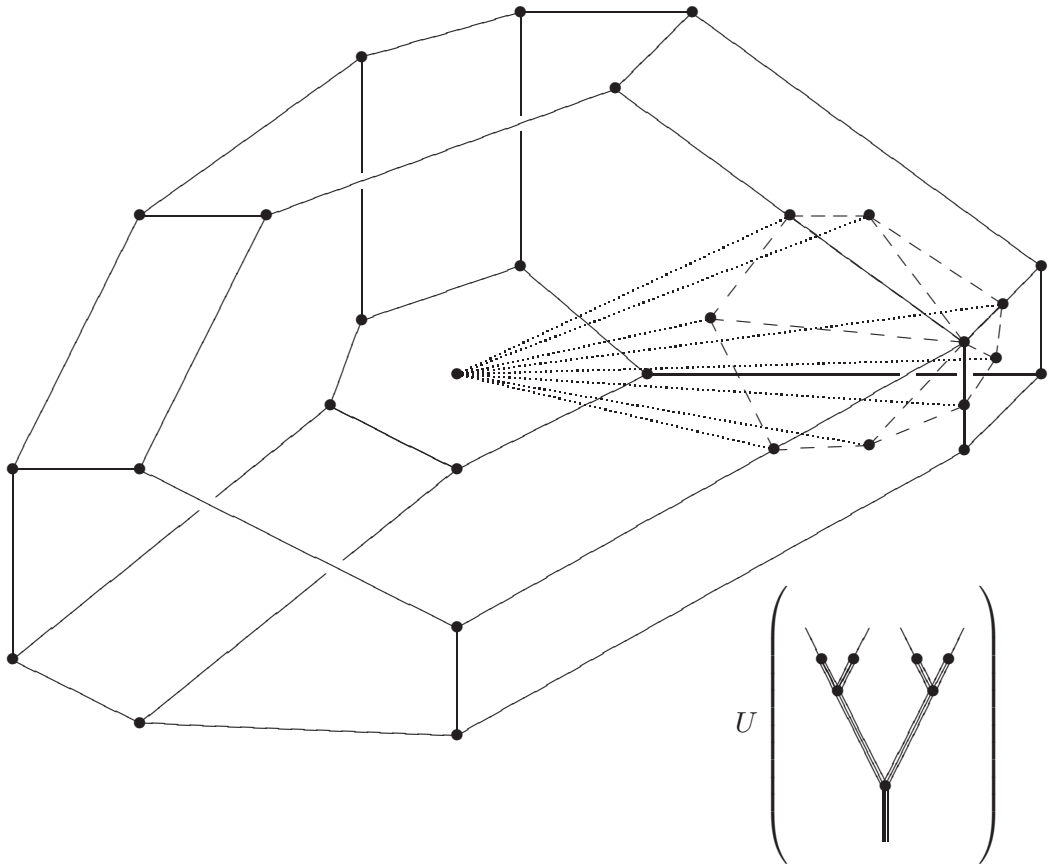
\begin{small}
$$
\xy 0;/r1pc/:
  (-4,0)*=0{\bullet}="1a";
  (0,-2)*=0{\bullet}="1b";
  (6,1)*=0{\bullet}="1c";
  (2,4.418)*=0{\bullet}="1d";
  (-3,2.7)*=0{\bullet}="1e";
  (-10,-2)*=0{\bullet}="2a";
  (0,-7)*=0{\bullet}="2b";
  (16,2)*=0{\bullet}="2c";
  (5,10)*=0{\bullet}="2d";
  (-6,6)*=0{\bullet}="2e";
  (-14,-2)*=0{\bullet}="3a";
  (-14,-8)*=0{\bullet}="3b";
  (-10,-10)*=0{\bullet}="3c";
  (0,-10.4)*=0{\bullet}="3d";
  (16,-1.418)*=0{\bullet}="3e";
  (18.418,1)*=0{\bullet}="3f";
  (18.418,4.418)*=0{\bullet}="3g";
  (7.418,12.418)*=0{\bullet}="3h";
  (2,12.418)*=0{\bullet}="3i";
  (-3,11)*=0{\bullet}="3j";
  (-10,6)*=0{\bullet}="3k";
 (14.222,1)*{\hole}="x"; (14.222,1)*{\hole}="x";
(16,1)*{\hole}="y";
  (-3,7.091)*{\hole}="z";
 (2,8.909)*{\hole}="w";
 (-7.846,-3.077)*{\hole}="p";
  (-3.846,-5.077)*{\hole}="q";
  "1a";"p" \ar@{-}; "1a";"p"\ar@{-};
  "p";"3b"  \ar@{-};
    "1b";"q" \ar@{-};
  "q";"3c" \ar@{-};
 "1c";"x" \ar@{-};
 "x";"y" \ar@{-};
 "y";"3f" \ar@{-};
 "1d";"w" \ar@{-};
 "w";"3i" \ar@{-};
 "1e";"z" \ar@{-};
 "z";"3j" \ar@{-};
 "2a";"3a" \ar@{-};
 "2b";"3d" \ar@{-};
  "2c";"3e" \ar@{-};
  "2c";"3g" \ar@{-};
 "2d";"3h" \ar@{-};
 "2e";"3k" \ar@{-};
"1a";"1b" \ar@{-}; "1a";"1b" \ar@{-};
 "1b";"1c" \ar@{-};
 "1c";"1d" \ar@{-};
 "1d";"1e" \ar@{-};
 "1e";"1a" \ar@{-};
 "2a";"2b" \ar@{-};
 "2b";"2c" \ar@{-};
 "2c";"2d" \ar@{-};
 "2d";"2e" \ar@{-};
 "2e";"2a" \ar@{-};
 "3a";"3b" \ar@{-};
  "3b";"3c" \ar@{-};
  "3c";"3d" \ar@{-};
  "3d";"3e" \ar@{-};
  "3e";"3f" \ar@{-};
  "3f";"3g" \ar@{-};
  "3g";"3h" \ar@{-};
  "3h";"3i" \ar@{-};
  "3i";"3j" \ar@{-};
  "3j";"3k" \ar@{-};
  "3k";"3a" \ar@{};
  %
  %
(8,2.75)*=0{\bullet}="u1";
 (10.5,6)*=0{\bullet}="u2";
 (13,6)*=0{\bullet}="u3";
 (17.209,3.209)*=0{\bullet}="u4";
 (17,1.5)*=0{\bullet}="u5";
 (16,0)*=0{\bullet}="u6";
 (13,-1.25)*=0{\bullet}="u7";
 (10,-1.375)*=0{\bullet}="u8";
 (0,1)*=0{\bullet}="u0";
"u1";"u2" \ar@{--};"u1";"u2" \ar@{--};
 "u2";"u3" \ar@{--};
 "u3";"u4" \ar@{--};
 "u4";"u5" \ar@{--};
 "u5";"u6" \ar@{--};
 "u6";"u7" \ar@{--};
 "u7";"u8" \ar@{--};
 "u1";"u8" \ar@{--};
"u1";"2c" \ar@{--};
 "u2";"2c" \ar@{--};
 "u3";"2c" \ar@{--};
 "u4";"2c" \ar@{--};
 "u5";"2c" \ar@{--};
 "u6";"2c" \ar@{--};
 "u7";"2c" \ar@{..};
 "u0";"u1" \ar@{..};
 "u0";"u2" \ar@{..};
 "u0";"u3" \ar@{..};
 "u0";"u4" \ar@{..};
 "u0";"u5" \ar@{..};
 "u0";"u6" \ar@{..};
 "u0";"u7" \ar@{..};
"u0";"u8" \ar@{};
 (11,-7)*=0{}="a"; (13,-7)*=0{}="b";
  (14,-7)*=0{}="c"; (16,-7)*=0{}="d";
  (12,-9)*=0{\bullet}="v1"; (15,-9)*=0{\bullet}="v2";
  (13.5,-12)*=0{\bullet}="v3";
  (13.5,-13.75)*=0{}="v4";
  (14.5,-8)*=0{\bullet}="va";
  (15.5,-8)*=0{\bullet}="vb";
  (12.5,-8)*=0{\bullet}="vc";
  (11.5,-8)*=0{\bullet}="vd"; \ar@{};
 "a" ;"v1" \ar@{-};"a" ;"v1" \ar@{-};
 "v1" ;"vc" \ar@{-};
 "b" ;"v1" \ar@{-}  ;
 "c" ;"va" \ar@{-} ;
 "d" ;"vb" \ar@{-} \ar@{};
 "va";"v2" \ar@{=} ;"va";"v2" \ar@{=};
 "vb";"v2" \ar@{=};
 "v2" ;"v3" \ar@{=};
 "vd" ;"v3" \ar@{=} ;
  "vc" ;"v1" \ar@{=} ;
 "v3" ;"v4" \ar@{=} \ar@{};
 "va";"v2" \ar@{-} ;"va";"v2" \ar@{-} ;
 "vb";"v3" \ar@{-} ;
 "vd" ;"v3" \ar@{-} ;
 "v3" ;"v4" \ar@{} ;
 (13,-10) *={U \left( \xy  0;/r.25pc/:
  (-10,18)*=0{}="a"; (-2,18)*=0{}="b";
  (2,18)*=0{}="c"; (14,18)*=0{~}="d";
  (-6,12)*=0{~}="v1"; (6,12)*=0{~}="v2";
  (0,0)*=0{~}="v3";
  (0,-5)*=0{}="v4";
 \endxy
  \right)}
    \endxy
$$
\end{small}
\end{proof}
\section{Proof of Theorem~\ref{main}}\label{proofsec}
 To demonstrate that our convex hulls are each combinatorially equivalent to the
  corresponding convex $CW$-complexes defined by Iwase and
 Mimura, we need only check that they both have the same vertex-facet incidence. We will show that for
 each $n$
 there is an isomorphism $f$ between the  vertex sets (0-cells) of our convex hull and ${\cal J}(n)$
 which preserves the sets of vertices
  corresponding to facets; i.e. if $S$ is the set of vertices of a facet of our convex hull then
   $f(S)$ is a vertex set of
  a facet of ${\cal J}(n).$

  To demonstrate the existence of the isomorphism, noting that the vertices of ${\cal J}(n)$
  correspond to the binary painted trees, we only
 need to check that the points we calculate from those binary painted trees are indeed the vertices of their convex hull.
 The isomorphism implied is the one that takes a vertex associated to a certain tree to the 0-cell
 associated to the same tree. Now a given facet of ${\cal J}(n)$ corresponds to a tree $T$ which is one of the
 two sorts of trees pictured in Definitions~\ref{ufacets} and~\ref{lfacets}.
 To show that our implied isomorphism of vertices preserves vertex sets of
facets we need to show  that for each $T$ there is one facet that is the convex hull of the points
corresponding to the binary trees which are \emph{refinements} of $T$. By refinement of painted
trees we refer to the relationship: $t$ refines $t'$ if $t'$ results from the collapse of some of
the internal edges of $t$.
Note that the two sorts of trees pictured in Definitions~\ref{ufacets} and~\ref{lfacets} are each a
single collapse away from being the painted corolla.


The proofs of both key points will proceed in tandem, and will be inductive.
 The main strategy will
be to define a dimension $n-2$ affine hyperplane $H_q(T)$ in {\bf R}$^{n-1}$ for each of  the upper
and lower facet trees $T$ (as drawn in the Definitions~\ref{ufacets} and~\ref{lfacets}), and then
to show that these are the proper bounding hyperplanes of the convex hull (i.e. that each actually
contains a facet). The definition of hyperplane will actually generalize our algorithm for finding
a point $M_q(t)$ in {\bf R}$^{n-1}$ from a binary tree $t$ with $n$ leaves. The proof of
Theorem~\ref{main} will however not use these hyperplanes directly, but recast them in a weighted
version. Then they will be recovered when the weights are all set equal to 1.
\begin{definition}
The lower facets ${\cal J}_k(r,s)$ correspond to lower trees such as:
 
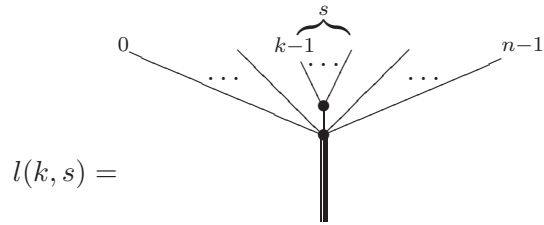
\begin{small}
 $$
\xy (0,0) *{l(k,s) =}\endxy
\xy (0,0) *{ \xy  0;/r.45pc/:
    (-8,22) *{{s \atop \overbrace{~~~~~~~~~}}};
   (-22,20)*{{}^0}="L";(6,20)*{{}^{n-1}}="R";
  (-14,20)*=0{}="a"; (-10,20)*{{}^{k-1}}="b";
  (-6,20)*=0{}="b2";
  (-2,20)*=0{}="c";
  (-8,14)*=0{\bullet}="v1";
  (-8,8)*=0{}="v3";
  (-8,16)*=0{\bullet}="v0";
  (-15,18) *{\dots};
  (-1,18) *{\dots};
  (-8,19) *{\dots};
   "a" ;"v1" \ar@{-};"a" ;"v1" \ar@{-};
   "L" ;"v1" \ar@{-};
   "R" ;"v1" \ar@{-};
 "b" ; "v0" \ar@{-};
 "b2" ; "v0" \ar@{-};
 "v0"; "v1" \ar@{-};
 "v1" ; "c" \ar@{-};
  "v1" ;"v3" \ar@{-};
  \ar@{} ;
"v1" ;"v3" \ar@{=};"v1" ;"v3" \ar@{=};
 \endxy} \endxy
$$
\end{small}

These are assigned a hyperplane $H_q(l(k,s))$ determined by the equation
$$x_k + \dots + x_{k+s-2}
= {q \over 2}s(s-1).$$

\end{definition}
 Recall that $r$ is the number of branches extending from the lowest node,
and $r+s = n+1.$ Thus $1\le k \le r.$ Notice that if $s = n$ (so $r= k =1$) then this becomes the
hyperplane given by
$$x_1 + \dots + x_{n-1} = {q \over 2}n(n-1) = qS(n-1).$$
 Therefore the points
$M_q(t)$ for $t$ a binary tree with only nodes type (1) and (3) will lie in the hyperplane
$H_q(l(1,n))$ by Lemma 2.5 of \cite{loday}. (Simply multiply both sides of the relation proven
there by $q$.) Also note that for $q=1$ (thus disregarding the painting) that these hyperplanes are
an alternate to the bounding hyperplanes of the associahedron defined by Loday using admissible
shuffles. Our hyperplanes (for $q=1$) each have the same intersection with the hyperplane $H$ as
does the corresponding hyperplane $H_{\omega}$ defined by Loday (for $\omega$ corresponding to the
unpainted version of our tree $l(k,s)$.)

\begin{definition}
The upper facets ${\cal J}(t;r_1,\dots ,r_t)$ correspond to upper trees such as:
 
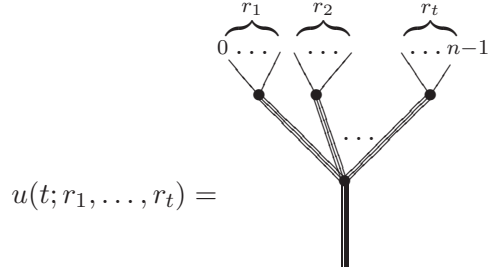
\begin{small}
 $$
\xy (0,0) *{ u(t;r_1,\dots, r_t)= }\endxy
\xy (0,0) *{ \xy  0;/r.45pc/:
  (-.5,23)*{{}^0}="a"; (1.5,25.25) *{{r_1 \atop \overbrace{~~~}}}; (1.5,23) *{\dots}; (3.5,23)*=0{}="b";
(4.5,23)*=0{}="a2"; (6.5,25.25) *{{r_2 \atop \overbrace{~~~}}}; (6.5,23)  *{\dots};
(8.5,23)*=0{}="b2";
 (12,23)*=0{}="a3"; (14,25.25) *{{r_t \atop \overbrace{~~~}}}; (14,23) *{\dots~~}; (16,23)*{~~^{~n-1}}="b3";
   (2,20)*=0{\bullet}="d";
  (6,20)*=0{\bullet}="e"; (14,20)*=0{\bullet}="f";
 (9,17) *{\dots};
  (8,14)*=0{\bullet}="v2";
  (8,8)*=0{}="v4"; \ar@{};
 "d" ;"v2" \ar@{-}; "d" ;"v2" \ar@{-};
"a" ;"d" \ar@{-};
 "b" ;"d" \ar@{-};
  "a2" ;"e" \ar@{-};
   "b2" ;"e" \ar@{-};
 "a3" ;"f" \ar@{-};
 "b3" ;"f" \ar@{-};
 "e" ;"v2" \ar@{-};
 "f" ;"v2" \ar@{-};
  "v2" ;"v4" \ar@{-};
 \ar@{} ;
"d" ;"v2" \ar@{=};"d" ;"v2" \ar@{=};
 "e" ;"v2" \ar@{=};
 "f" ;"v2" \ar@{=};
 "v2" ;"v4" \ar@{=};
 \endxy} \endxy
$$
\end{small}

These are assigned a hyperplane $H_q(u(t;r_1,\dots, r_t))$ determined by the equation
$$x_{r_1} + x_{(r_1+r_2)} +
\dots + x_{(r_1+r_2+\dots +r_{t-1})} = {1 \over 2}\left(n(n-1)-\sum_{i=1}^t r_i(r_i-1)\right)$$ or
equivalently:
 $$x_{r_1} + x_{(r_1+r_2)} + \dots + x_{(r_1+r_2+\dots +r_{t-1})} = \sum_{1\le i<j \le t}r_ir_j.$$
  \end{definition}

 Note that if $t=n$ (so $r_i = 1$ for all $i$) that
this becomes the hyperplane given by
$$x_1 + \dots +
x_{n-1} = {1 \over 2}n(n-1) = S(n-1).$$ Therefore the points $M_q(t)$ for $t$ a binary tree with
only nodes type (2) and (3) will lie in the hyperplane $H$ by Lemma 2.5 of \cite{loday} (using
notation $S(n)$ and $H$ as in that source).

  In order to prove Theorem~\ref{main} it turns out to be expedient to prove a more general result. This
  consists of an even more flexible version of the algorithm for assigning points to binary trees
  in order to achieve a convex hull of those points which is the multiplihedron. To assign points
  in {\bf R}$^{n-1}$
  to the binary painted trees with $n$ leaves, we not only choose a value $q \in (0,1)$ but also
  an ordered $n$-tuple of positive integers $w_0, \dots, w_{n-1}.$  Now given a tree $t$
we calculate a point $M^{w_0, \dots, w_{n-1}}_q(t)$ in {\bf R}$^{n-1}$ as follows: we begin by
assigning the weight $w_i$ to the $i^{th}$ leaf. We refer to the result as a weighted tree. Then we
modify Loday's algorithm for finding the coordinate for each trivalent node by replacing the number
of leaves of the left and right subtrees with the sums of the weights of the leaves of those
subtrees. Thus we let $L_i = \sum w_k$ where the sum is over the leaves of the subtree supported by
the left branch of the $i^{th}$ node. Similarly we let $R_i = \sum w_k$ where $k$ ranges over the
leaves of the the subtree supported by the right branch.
  Then
$$M^{w_0, \dots, w_{n-1}}_q(t) = (x_1, \dots x_{n-1}) \text{ where } x_i = \begin{cases}
    qL_iR_i,& \text{if node $i$ is type (1)} \\
    L_iR_i,& \text{if node $i$ is type (2).}
\end{cases} $$
Note that the original points $M_q(t)$ are recovered if $w_i = 1$ for  $i=0,\dots,n-1.$ Thus
proving that the convex hull of the points $M^{w_0, \dots, w_{n-1}}_q(t)$ where $t$ ranges over the
binary painted trees with $n$ leaves is the $n^{th}$ multiplihedron will imply the main theorem.
For an example,  let us calculate the point in {\bf R}$^3$ which corresponds to the 4-leaved tree:

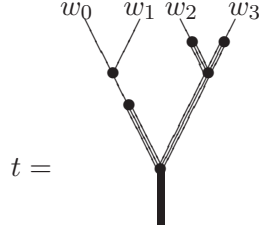
\begin{small}
 $$ t =
\xy  0;/r.25pc/:
  (-10,20)*{ w_0~}="a"; (-2,20)*{w_1~}="b";
  (2,20)*{~w_2}="c"; (10,20)*{~w_3}="d";
  (-6,12)*=0{\bullet}="v1"; (6,12)*=0{\bullet}="v2";
  (0,0)*=0{\bullet}="v3";
  (0,-7)*=0{}="v4";
  (4,16)*=0{\bullet}="va";
  (8,16)*=0{\bullet}="vb";
  (-4,8)*=0{\bullet}="vc" \ar@{};
 "a" ;"v1" \ar@{-};"a" ;"v1" \ar@{-};
 "v1" ;"vc" \ar@{-};
 "b" ;"v1" \ar@{-}  ;
 "c" ;"va" \ar@{-} ;
 "d" ;"vb" \ar@{-} \ar@{};
 "va";"v2" \ar@{=} ;"va";"v2" \ar@{=} ;
 "vb";"v2" \ar@{=} ;
 "v2" ;"v3" \ar@{=};
 "vc" ;"v3" \ar@{=} ;
 "v3" ;"v4" \ar@{=} \ar@{};
 %
 %
 "va";"v2" \ar@{-} ;"va";"v2" \ar@{-} ;
 "vb";"v3" \ar@{-} ;
 "vc" ;"v3" \ar@{-} ;
 "v3" ;"v4" \ar@{-} ;
 \endxy
$$
\end{small}

Now $M^{w_0, \dots, w_{3}}_q(t) = (qw_0w_1, (w_0+w_1)(w_2+w_3), w_2w_3). $
 To motivate this new weighted version of our algorithm we mention that the weights
 $w_0,\dots,w_{n-1}$ are to be thought of as the sizes of various trees to be grafted to the
 respective leaves. This weighting is therefore necessary to make the induction go through, since
 the induction is itself based upon the grafting of trees.

 \begin{lemma}\label{weight}
 For $q=1$ the convex hull of the points $M^{w_0, \dots, w_{n-1}}_q(t)$ for $t$ an $n$-leaved binary tree
 gives the $n^{th}$ associahedron.
 \end{lemma}
 \begin{proof}
 Recall that for $q=1$ we can ignore the painting, and thus for $w_i = 1$ for $i=0,\dots,n-1$
 the points we calculate are exactly those
 calculated by Loday's algorithm.  Now for arbitrary weights $w_0, \dots, w_{n-1}$ we can form from each
 weighted tree $t$ (with those weights assigned to the respective leaves)  a non-weighted tree $t'$  formed by grafting
 a corolla with $w_i$ leaves onto the $i^{th}$ leaf of $t.$ Note that for binary trees which are refinements of $t'$
 the coordinates which correspond to the nodes of $t'$ below the grafting receive precisely
  the same value from Loday's algorithm
 which the  corresponding nodes of the original weighted tree received  from the weighted algorithm.
 Now since Loday's algorithm gives the
 vertices of the associahedra, then the binary trees which are refinements of $t'$ give the vertices of
 ${\cal K}(n) \times {\cal K}(w_0)\times \dots \times {\cal K}(w_{n-1}).$ If we restrict our attention in each entire
  binary refinement of $t'$ to the
  nodes of (the refinements of) the grafted
 corolla with $w_i$ leaves we find the vertices of ${\cal K}(w_i).$ The definition of a cartesian product of polytopes
 guarantees that the vertices of the product are points which are cartesian products of the vertices of the operands.
 Polytopes are also combinatorially invariant under change of basis, and so we can rearrange the
 coordinates of our vertices to put all the coordinates corresponding to the nodes of (the refinements of) the grafted
 corollas at the end of the point, leaving the coordinates corresponding to the nodes below the
 graft in order at the beginning of the point.
 Thus the nodes below
 the grafting correspond to the vertices of ${\cal K}(n),$ and so the weighted algorithm (with $q=1$) does give
 the vertices of ${\cal K}(n).$
 \end{proof}

\begin{lemma}\label{weightcount}
For $q=1$ the points $M^{w_0, \dots, w_{n-1}}_q(t)$ for $t$ an $n$-leaved binary tree all lie in
the $n-2$ dimensional affine hyperplane of {\bf R}$^{n-1}$ given by the equation:
$$x_1 + \dots + x_{n-1} = \sum_{1\le i<j \le (n-1)}w_iw_j.$$
\end{lemma}
\begin{proof}
In Lemma 2.5 of \cite{loday} it is shown inductively that when $w_i = 1$ for $i=1,\dots, n-1$ then
the point $M^{1,\dots, 1}_1(t)=M(t)=(x_1,\dots,x_{n-1})$ satisfies the equation
$\sum_{i=1}^{n-1}x_i = {1 \over 2}n(n-1).$ As in the proof of the previous lemma we replace the
weighted tree $t$ with the non-weighted $t'$ formed by grafting an arbitrary binary tree with $w_i$
leaves to the $i^{th}$ leaf of $t.$ Let $m = \sum_{i=1}^{n-1}w_i.$ Thus the point $M^{1,\dots,
1}_1(t')=M(t')=(x_1,\dots,x_{m})$ satisfies the equation
$$\sum_{i=1}^{m-1}x_i = {1 \over 2}m(m-1)= {1 \over 2}\sum_{i=1}^{n-1} w_i(\sum_{i=1}^{n-1} w_i-1).$$
Also the coordinates corresponding to the nodes of the grafted tree with $w_i$ leaves sum up to the
value ${1 \over 2}w_i(w_i-1).$ Thus the coordinates corresponding to the nodes below the graft,
that is, the coordinates of the original weighted tree $t$, sum up to the difference:
$${1 \over 2}\left(\sum_{i=1}^{n-1} w_i(\sum_{i=1}^{n-1} w_i-1)-\sum_{i=1}^{n-1} w_i(w_i-1)\right) =  \sum_{1\le i<j \le (n-1)}w_iw_j$$

\end{proof}

 Since we are proving that the points $M^{w_0, \dots, w_{n-1}}_q(t)$ are the vertices of the
multiplihedron, we need to define hyperplanes  $H^{w_0, \dots, w_{n-1}}_q(t)$ for this weighted
version which we will show to be the the bounding hyperplanes when $t$ is a facet tree.
\begin{definition}
\end{definition}
Recall that the lower facets ${\cal J}_k(r,s)$ correspond to lower trees such as:
 \begin{small}
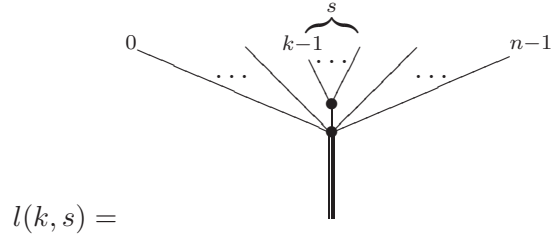
 $$
l(k,s) = \xy  0;/r.45pc/:
    (-8,22) *{{s \atop \overbrace{~~~~~~~~~}}};
   (-22,20)*{{}^0}="L";(6,20)*{{}^{n-1}}="R";
  (-14,20)*=0{}="a"; (-10,20)*{{}^{k-1}}="b";
  (-6,20)*=0{}="b2";
  (-2,20)*=0{}="c";
  (-8,14)*=0{\bullet}="v1";
  (-8,8)*=0{}="v3";
  (-8,16)*=0{\bullet}="v0";
  (-15,18) *{\dots};
  (-1,18) *{\dots};
  (-8,19) *{\dots};
   "a" ;"v1" \ar@{-};"a" ;"v1" \ar@{-};
   "L" ;"v1" \ar@{-};
   "R" ;"v1" \ar@{-};
 "b" ; "v0" \ar@{-};
 "b2" ; "v0" \ar@{-};
 "v0"; "v1" \ar@{-};
 "v1" ; "c" \ar@{-};
  "v1" ;"v3" \ar@{-};
  \ar@{} ;
"v1" ;"v3" \ar@{=};"v1" ;"v3" \ar@{=};
 \endxy
$$
\end{small}

These are assigned a hyperplane $H^{w_0, \dots, w_{n-1}}_q(l(k,s))$ determined by the equation
$$x_k + \dots + x_{k+s-2} = q\left(\sum_{(k-1)\le i<j \le (k+s-2)}w_iw_j\right).$$
Recall that $r$ is the number of branches from the lowest node, and $r+s = n+1.$
\newline
\begin{lemma}\label{low}
For any painted binary tree $t$ the point $M^{w_0, \dots, w_{n-1}}_q(t)$ lies in the hyperplane
$H^{w_0, \dots, w_{n-1}}_q(l(k,s))$ iff $t$ is a refinement of $l(k,s).$ Also the hyperplane
$H^{w_0, \dots, w_{n-1}}_q(l(k,s))$ bounds below the points $M^{w_0, \dots, w_{n-1}}_q(t)$ for $t$
any binary painted tree.
\end{lemma}
\begin{proof}
By Lemma~\ref{weightcount} we have that any binary tree $t$ which is a refinement of the lower tree
$l(k,s)$ will yield a point $M^{w_0, \dots, w_{n-1}}_q(t)$ which lies in $H^{w_0, \dots,
w_{n-1}}_q(l(k,s)).$ To see this we simply note that the nodes in $t$ associated to the coordinates
$x_k , \dots , x_{k+s-2}$ in $M^{w_0, \dots, w_{n-1}}_q(t)$ will each be of type (1), and so we
multiply by $q$ on both sides of the equation proven in the Lemma.
\newline
We now demonstrate that if a binary tree $t$ is not a refinement of a lower tree $l(k,s)$ then the
point $M^{w_0, \dots, w_{n-1}}_q(t)$ will have the property that
$$x_k + \dots + x_{k+s-2} > q\left(\sum_{(k-1)\le i<j \le (k+s-2)}w_iw_j\right).$$
Recall that the trees which are refinements of $l(k,s)$ have all their nodes  inclusively between
$k$ and $k+s-2$ of type (1).  Now if  $t$ has these same $s-1$ nodes $k, \dots, k+s-2$ all type (1)
and is not a refinement of $l(k,s)$ then there is no node in $t$ whose deletion results in the
separation of only the leaves $k-1, \dots, k+s-2$ from the rest of the leaves of $t.$ Let $t'$ be
the subtree of $t$ determined by taking as its root the node furthest from the root of $t$ whose
deletion results in the separation of all the leaves $k-1, \dots, k+s-2$ from the rest of the
leaves of $t.$ Thus $t'$ will have more than just those $s$ leaves, say those leaves of $t$ labeled
$k-p,\dots, k+p'-2$ where $p\ge 1, ~ p' \ge s$ and at least one of the inequalities strict. Since
the situation is symmetric we just consider the case where $p'=s$ and $p>1.$ Then we have an
expression for the sum of all the coordinates whose nodes are in $t'$ and can write:
$$(*)~\txt{~~~~}~~x_k + \dots + x_{k+s-2} = q\left(\sum_{(k-p)\le i<j \le (k+s-2)}w_iw_j\right) - q(x_{k-p+1} + \dots + x_{k-1}).$$
Notice that the first sum on the right hand side of $(*)$ contains
$$ x_{k-p+1} + \dots + x_{k-1}+
\sum_{(k-1)\le i<j \le (k+s-2)}w_iw_j .$$ (There is no overlap between the coordinate values here
and the sum since each of the terms in $x_{k-p+1} + \dots + x_{k-1}$ contains a factor from
$w_{k-p} , \dots , w_{k-2}.$) The first sum on the right hand side of $(*)$ also contains at least
one term $w_mw_j$ where $(k-p)\le m \le (k-2) $ and where $w_mw_j$ does not occur as a term in
$x_{k-p+1} + \dots + x_{k-1},$ else the leaf labeled by $m$ would not lie in $t'.$ Thus we have the
desired inequality. Here is a picture of an example situation, where $p=2.$ Note that the key term
$w_mw_j$ in the above discussion is actually $w_{k-2}w_{k+1}$ in this picture.
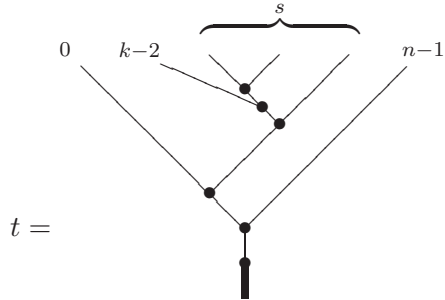
\begin{small}
 $$
t = \xy  0;/r.55pc/:
    (2,12) *{{s \atop \overbrace{\text{\hspace{60pt}}}}};
   (-10,10)*{{}^0~}="L";(10,10)*{{~}^{n-1}}="R";
  (-6,10)*{{}^{k-2}}="a"; (-2,10)*=0{}="b";
  (2,10)*=0{}="b2";
  (6,10)*=0{}="c";
  (0,8)*=0{\bullet}="v1";
  (1,7)*=0{\bullet}="v2";
  (2,6)*=0{\bullet}="v3";
  (-2,2)*=0{\bullet}="v4";
  (0,0)*=0{\bullet}="v5";
  (0,-2)*=0{\bullet}="v6";
  (0,-4)*=0{}="v0";
   "a" ;"v2" \ar@{-};"a" ;"v2" \ar@{-};
   "L" ;"v5" \ar@{-};
   "R" ;"v5" \ar@{-};
 "b" ; "v3" \ar@{-};
 "b2" ; "v1" \ar@{-};
 "v5"; "v6" \ar@{-};
 "v3" ; "c" \ar@{-};
 "v3" ; "v4" \ar@{-};
  "v6" ;"v0" \ar@{-};
  \ar@{} ;
"v6" ;"v0" \ar@{=};"v6" ;"v0" \ar@{=};
 \endxy
$$
\end{small}
Now if in the situation for which there does not exist a node of $t$ which if deleted would
separate exactly the  leaves $k-1,\dots,k+s-2$ from the other leaves and root of $t,$ there are
also some of the nodes in $k, \dots, k+s-1$ of type (2), the inequality still holds, and now to a
greater degree since some of the factors of $q$ are missing from the right hand side.

If there does exist a node of $t$ which if deleted would separate exactly the  leaves
$k-1,\dots,k+s-2$ from the other leaves and root of $t,$ but $t$ is not a refinement of $l(k,s)$
due to the painting (some of the nodes in $k, \dots, k+s-1$ are of type (2)), then the inequality
holds precisely because the only difference left to right is that the right hand side has fewer
terms multiplied by the factor of $q.$

\end{proof}
\begin{definition}
\end{definition}
Recall that the upper facets ${\cal J}(t;r_1,\dots ,r_t)$ correspond to upper trees such as:
 \begin{small}
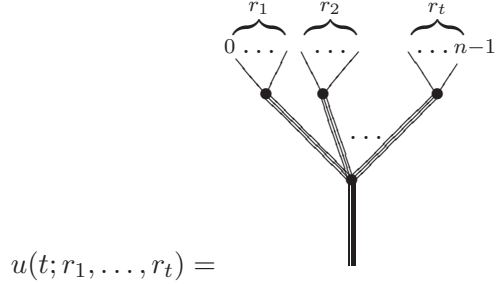
 $$
u(t;r_1,\dots, r_t) = \xy  0;/r.45pc/:
  (-.5,23)*{{}^0}="a"; (1.5,25.25) *{{r_1 \atop \overbrace{~~~}}}; (1.5,23) *{\dots}; (3.5,23)*=0{}="b";
(4.5,23)*=0{}="a2"; (6.5,25.25) *{{r_2 \atop \overbrace{~~~}}}; (6.5,23)  *{\dots};
(8.5,23)*=0{}="b2";
 (12,23)*=0{}="a3"; (14,25.25) *{{r_t \atop \overbrace{~~~}}}; (14,23) *{\dots~~}; (16,23)*{~~^{~n-1}}="b3";
   (2,20)*=0{\bullet}="d";
  (6,20)*=0{\bullet}="e"; (14,20)*=0{\bullet}="f";
 (9,17) *{\dots};
  (8,14)*=0{\bullet}="v2";
  (8,8)*=0{}="v4"; \ar@{};
 "d" ;"v2" \ar@{-}; "d" ;"v2" \ar@{-};
"a" ;"d" \ar@{-};
 "b" ;"d" \ar@{-};
  "a2" ;"e" \ar@{-};
   "b2" ;"e" \ar@{-};
 "a3" ;"f" \ar@{-};
 "b3" ;"f" \ar@{-};
 "e" ;"v2" \ar@{-};
 "f" ;"v2" \ar@{-};
  "v2" ;"v4" \ar@{-};
 \ar@{} ;
"d" ;"v2" \ar@{=};"d" ;"v2" \ar@{=};
 "e" ;"v2" \ar@{=};
 "f" ;"v2" \ar@{=};
 "v2" ;"v4" \ar@{=};
 \endxy
$$
\end{small}

These are assigned a hyperplane $H^{w_0, \dots, w_{n-1}}_q(u(t;r_1,\dots, r_t))$   determined by
the equation
 $$x_{r_1} + x_{(r_1+r_2)} + \dots + x_{(r_1+r_2+\dots +r_{t-1})} = \sum_{1\le i<j \le t}R_iR_j.$$
where $R_i = \sum_{}w_j$ where the sum is over the leaves of the $i^{th}$ subtree  (from left to
right) with root the type (5) node; the index $j$ goes from $(r_1+r_2+\dots +r_{i-1})$ to
$(r_1+r_2+\dots +r_{i}-1)$ (where $r_0 = 0.$)
  Note that if $t=n$ (so $r_i = 1$ for all $i$) that
this becomes the hyperplane given by
$$x_1 + \dots +
x_{n-1} = \sum_{1\le i<j \le n-1}w_iw_j.$$
\newline
\begin{lemma}\label{up}
For any painted binary tree $t$ the point $M^{w_0, \dots, w_{n-1}}_q(t)$ lies in the hyperplane
$H^{w_0, \dots, w_{n-1}}_q(u(t;r_1,\dots, r_t))$ iff $t$ is a refinement of $u(t;r_1,\dots, r_t).$
Also the hyperplane $H^{w_0, \dots, w_{n-1}}_q(u(t;r_1,\dots, r_t))$ bounds above the points
$M^{w_0, \dots, w_{n-1}}_q(t)$ for $t$ any binary painted tree.
\end{lemma}
\begin{proof}
Now by by Lemma~\ref{weightcount} we have that any binary tree $t$ which is a refinement of the
upper tree $u(t;r_1,\dots, r_t)$ will yield a point $M^{w_0, \dots, w_{n-1}}_q(t)$ which lies in
$H^{w_0, \dots, w_{n-1}}_q(u(t;r_1,\dots, r_t)).$ To see this we simply note that the  coordinates
$x_{r_1} , x_{(r_1+r_2)} , \dots , x_{(r_1+r_2+\dots +r_{t-1})}$ in $M^{w_0, \dots, w_{n-1}}_q(t)$
will each be  assigned the same value as if the original upper tree had had $r_i = 1$ for all $i$
but where the weights given were $R_0, \dots R_{n-1}$.
\newline

We now demonstrate that if a binary tree $T$ is not a refinement of an upper tree
$u(t;r_1,\dots,r_t)$ then  the point $M^{w_0, \dots, w_{n-1}}_q(T)$ will have the property that
$$x_{r_1} + x_{(r_1+r_2)} + \dots + x_{(r_1+r_2+\dots +r_{t-1})} < \sum_{1\le i<j \le t}R_iR_j.$$
Recall that $R_i = \sum_{j}w_j$ where the sum is over the leaves of the $i^{th}$ subtree  (from
left to right) with root the type (5) node; the index $j$ goes from $(r_1+r_2+\dots +r_{i-1})$ to
$(r_1+r_2+\dots +r_{i}-1)$ (where $r_0 = 0.$) If $T$ is not a refinement of $u(t;r_1,\dots,r_t)$
then for some of the partitioned sets of $r_i$ leaves in the partition $r_1,\dots ,r_t$ it is true
that there does not exist a node of $T$ which if deleted would separate exactly the leaves in that
set from the other leaves and root of $T$. Thus the proof here will use the previous result for the
lower trees. First we consider the case for which $T$ is entirely painted--it has only type (2)
nodes. Now by Lemma~\ref{weightcount} the total sum of the coordinates of $M^{w_0, \dots,
w_{n-1}}_q(T)$ will be equal to $\sum_{1\le i < j \le n-1}w_iw_j.$ Consider a (partitioned) set of
$r_m$ leaves (starting with leaf $k-1$ ) in the partition $r_1,\dots ,r_t$ for which there does not
exist a node of $T$ which if deleted would separate exactly the leaves in that set from the other
leaves and root of $T.$ (Here $k-1 = r_1+r_2+\dots +r_{m-1}$) Let $P_m$ be the sum of the $r_m -1$
coordinates $x_{k}+\dots+x_{k+r_m-2}.$ We have by the same argument used for lower trees that
$$P_m > \sum_{(k-1)\le i<j \le(k+r_m-2)} w_iw_j.$$

 Now for this $T$,  for which some of the partitioned sets of $r_i$ leaves in the partition
  $r_1,\dots ,r_t$  there does not exist a node of $T$ which if deleted would separate exactly the leaves in that
set from the other leaves and root of $T$, we have:

$$x_{r_1} + x_{(r_1+r_2)} + \dots + x_{(r_1+r_2+\dots +r_{t-1})} =  \sum_{1\le i < j \le n-1}w_iw_j -
\sum_{m=1}^t P_m < \sum_{1\le i<j \le t}R_iR_j.$$

If a tree $T'$ has the same branching structure as $T$ but with some nodes of type (1) then the
argument still holds since the argument from the lower trees still applies. Now for a tree $T$
whose branching structure is a refinement of the branching structure of  the upper tree
$u(t;r_1,\dots,r_t)$, but  which has some of its nodes $r_1 , (r_1+r_2), \dots , (r_1+r_2+\dots
+r_{t-1})$ of type (1), the inequality holds simply due to the application of some factors $q$ on
the left hand side.
\end{proof}

\begin{proof} of Theorem~\ref{main}:
Now we may proceed with our inductive argument. The base case of $n = 2$ leaves is trivial to
check. The points in {\bf R}$^{1}$ are $w_0w_1$ and $qw_0w_1.$ Their convex hull is a line segment,
combinatorially equivalent to ${\cal J}(2).$ Now we assume that for all $i<n$ and for arbitrary
$q\in (0,1)$ and for positive integer weights $w_0,\dots, w_{i-1},$ that the convex hull of the
points $\{M^{w_0, \dots, w_{i-1}}_q(t)~|~ t \text{ is a painted binary tree with $i$ leaves}\}$ in
{\bf R}$^{i-1}$ is combinatorially equivalent to the complex ${\cal J}(i),$ and that the points
$M^{w_0, \dots, w_{i-1}}_q(t)$ are the vertices of the convex hull. Now for $i=n$ we need to show
that the equivalence still holds. Recall that the two items we plan to demonstrate are that the
points $M^{w_0, \dots, w_{n-1}}_q(t)$ are the vertices of their convex hull and that the facet of
the convex hull corresponding to a given lower or upper tree $T$ is the convex hull of just the
points corresponding to the binary trees that are refinements of $T.$ The first item will be seen
in the process of checking the second.

Given an $n$-leaved lower tree $l(k,s)$ we have from Lemma~\ref{low} that the points corresponding
to binary refinements of $l(k,s)$ lie in an $n-2$ dimensional hyperplane $H^{w_0, \dots,
w_{n-1}}_q(l(k,s))$ which bounds the entire convex hull. To see that this hyperplane does indeed
contain a facet of the entire convex hull we use the induction hypothesis to show that the
dimension of the convex hull of just the points in $H^{w_0, \dots, w_{n-1}}_q(l(k,s))$ is $n-2.$
Recall that the tree $l(k,s)$ is the result of grafting an unpainted $s$-leaved corolla onto leaf
$k-1$ of an $r$-leaved painted corolla. Thus the points $M^{w_0, \dots, w_{n-1}}_q(t)$ for $t$ a
refinement of $l(k,s)$ have coordinates $x_k,\dots,x_{k+s-1}$ which are precisely those of the
associahedron ${\cal K}(s),$  by Lemma~\ref{weight} (after multiplying by $q$). Now considering the
remaining coordinates, we see by induction that they are the coordinates of the multiplihedron
${\cal J}(r).$ This is by process of considering their calculation as if performed on an $r$-leaved
weighted tree $t'$ formed by replacing the subtree of $t$ (with leaves $x_{k-1},\dots,x_{k+s-1}$)
with a single leaf of weight $\sum_{j=k-1}^{k+s-1} w_j.$ Now after a change of basis to reorder the
coordinates, we see that the points corresponding to the binary refinements of $l(k,s)$ are the
vertices of a polytope combinatorially equivalent to ${\cal J}(r) \times {\cal K}(s)$ as expected.
Since $r+s = n+1$ this polytope has dimension $r-1 + s - 2  = n -2,$  and so is a facet of the
entire convex hull.

Given an $n$-leaved upper tree $u(t,r_1,\dots,r_t)$ we have from Lemma~\ref{up} that the points
corresponding to binary refinements of $u(t,r_1,\dots,r_t)$ lie in an $n-2$ dimensional hyperplane
$H^{w_0, \dots, w_{n-1}}_q(u(t,r_1,\dots,r_t))$ which bounds the entire convex hull. To see that
this hyperplane does indeed contain a facet of the entire convex hull we use the induction
hypothesis to show that the dimension of the convex hull of just the points in
 $H^{w_0,\dots,w_{n-1}}_q(u(t,r_1,\dots,r_t))$ is $n-2.$ Recall that the tree $u(t,r_1,\dots,r_t)$ is the
result of grafting painted $r_i$-leaved corollas onto leaf $i$ of a $t$-leaved completely painted
corolla. Thus the points $M^{w_0, \dots, w_{n-1}}_q(t)$ for $T$ a refinement of
$u(t,r_1,\dots,r_t)$ have coordinates corresponding to the nodes in the $i^{th}$ subtree which are
precisely those of the multiplihedron ${\cal J}(r_i),$  by the inductive hypothesis. Now
considering the remaining coordinates, we see by Lemma~\ref{weight} that they are the coordinates
of the associahedron ${\cal K}(t).$ This is by process of considering their calculation as if
performed on an $t$-leaved weighted tree $T'$ formed by replacing each (grafted) subtree of $T$
(with $r_i$ leaves)
 with a single leaf of weight $\sum_j w_j,$ where the sum is over the $r_i$ leaves of the $i^{th}$ grafted subtree.
  Now after
a change of basis to reorder the coordinates, we see that the points corresponding to the binary
refinements of $u(t,r_1,\dots,r_t)$ are the vertices of a polytope combinatorially equivalent to
${\cal K}(t) \times {\cal J}(r_1) \times \dots \times {\cal J}(r_t)$ as expected. Since $r_1+ \dots
+ r_t = n$ this polytope has dimension $t-2 + (r_1-1) + (r_2-1) + \dots + (r_t-1) = n -2,$ and so
is a facet of the entire convex hull.

Since  each $n$-leaved binary painted tree is a refinement of some upper and or or lower trees,
then the point associated to that tree is found as a vertex of some of the facets of the entire
convex hull, and thus is a vertex of the convex hull. This completes the proof. Recall that in
Lemma~\ref{lw} we have already shown that our convex hull is homeomorphic to the space of painted
trees $LW{\cal U}(n).$
\end{proof}
 A
picture of the convex hull giving ${\cal J}(4)$ is also available at

\url{http://faculty.tnstate.edu/sforcey/ct06.htm}.

The convex hull for ${\cal J}(5)$ with 80 vertices is also pictured  there as a Schlegel diagram
generated by polymake.

\end{document}